\documentclass[12pt]{amsart}
\usepackage{graphicx}
\usepackage{rotate,epic,eepic,epsf,epsfig}

\usepackage{amsmath}
\usepackage{amsfonts}
\usepackage{amssymb}

\usepackage{graphicx}
\usepackage{rotate,epic,eepic,epsf,epsfig}

\allowdisplaybreaks[1]

\newtheorem{thm}{Theorem}[section]
\newtheorem{rema}{Remark}[section]
\newtheorem{exple}{Example}[section]
\newtheorem{cor}[thm]{Corollary}
\newtheorem{lem}[thm]{Lemma}
\newtheorem{prop}[thm]{Proposition}

\newtheorem{defn}[thm]{Definition}
\theoremstyle{definition}

\numberwithin{equation}{section}\newcommand{\resumename}{R\'esum\'e}

\DeclareFixedFont{\petitefonte}{\encodingdefault}%
{\familydefault}{\seriesdefault}{\shapedefault}{5pt}

\begin{document}


\title[Diamond cone for $\mathfrak{sl}(m,n)$]{Diamond cone for $\mathfrak{sl}(m,n)$}
\author[B. Agrebaoui, D. Arnal and O. Khlifi]{Boujema\^{a} Agrebaoui, Didier Arnal, Olfa Khlifi}

\address{Unit\'e de recherche UR 09-06, Facult\'e des Sciences, Universit\'e de Sfax, Route de Soukra, km 3,5,
B.P. 1171, 3000 Sfax, Tunisie.} \email{<B.Agreba@fss.rnu.tn>;}

\address{
Institut de Math\'ematiques de Bourgogne\\
UMR CNRS 5584\\
Universit\'e de Bourgogne\\
U.F.R. Sciences et Techniques
B.P. 47870\\
F-21078 Dijon Cedex\\France} \email{Didier.Arnal@u-bourgogne.fr}

\address{Unit\'e de recherche UR 09-06, Facult\'e des Sciences, Universit\'e de Sfax, Route de Soukra, km 3,5,
B.P. 1171, 3000 Sfax, Tunisie.} \email{khlifi\_olfa@yahoo.fr}

\begin{abstract}
In this paper, we first study the shape algebra and the reduced shape algebra for the Lie superalgebra $\mathfrak{sl}(m,n)$. We define the quasistandard tableaux, their collection is the diamond cone for $\mathfrak{sl}(m,n)$, which is a combinatorial basis for the reduced shape algebra. We realize a bijection between the set of semistandard tableaux with shape $\lambda$ and the set of quasistandard tableaux with shape $\mu\leq\lambda$, by using the `super jeu de taquin' on skew semistandard tableaux. This gives the compatibility of the diamond cone with the natural stratification of the reduced shape algebra.
\end{abstract}





\maketitle \vspace{.20cm}
\vskip1cm

\date{\today}


\section{Introduction}


\

The theory of Lie superalgebras was initied by V. Kac (see for instance \cite{K}). V. Kac introduced the classical (simple) Lie superalgebras, studied their classification and their irreducible (finite dimensional) representations, which are characterized by their highest dominant weight $\lambda$. Denote the corresponding module $\mathbb S^\lambda$. The theory is thus very close to the usual theory for simple Lie algebras. However, there are important differences. For instance the tensor product of two irreducible representations can be not completely reducible.\\

Denote $\mathbb S^\bullet$ the direct sum of all irreducible representations of a simple Lie algebra $\mathfrak g$. It is a natural algebra called the shape algebra of $\mathfrak g$. For instance, consider $\mathfrak g=\mathfrak{sl}(m)$, put $V=\mathbb C^m$. The space $V^{\otimes N}$ is a completely reducible $\mathfrak{gl}(m)$ representation. An explicit decomposition is given by the Schur-Weyl duality: let $\lambda=(a_1,\dots,a_m)$ be a sequence of natural numbers such that $N=\sum_jja_j$. One can see $\lambda$ as the shape of the Young tableaux having $a_j$ columns with height $j$. For any standard tableau $S_\lambda$ with shape $\lambda$, define the Young symmetrizer $c_{S_\lambda}$ in the group algebra of $S_N$, acting on the right side on tensors. As a consequence of the duality, the space $V_{S_\lambda}=\{v\in V^{\otimes N},~vc_{S_\lambda}=v\}$ is a simple $\mathfrak{gl}(m)$ module whose type is characterized by $\lambda$. By varying $N$ and $\lambda$, we get all the simple $\mathfrak{gl}(m)$ modules. For restriction to $\mathfrak{sl}(m)$, we consider only $\lambda$ such that $a_m=0$ and the standard tableau $St_\lambda$ obtained by the filling column by column. Thus, the shape algebra $\mathbb S^\bullet$ is a quotient of $Sym^\bullet(\wedge V)$. Moreover, a basis for this shape algebra is indexed by the family of semistandard Young tableaux with at most $m-1$ rows.\\

In this paper we look for the Lie superalgebra $\mathfrak{sl}(m,n)$ (see \cite{K,HKTV,V}). Put $V=\mathbb C^{m,n}$. The definition of the shape algebra for $\mathfrak{sl}(m,n)$ requires the restriction to the irreducible covariant tensorial representations, they are the irreducible subrepresentations in $T(V)$, indeed, the tensor product of two such representations is completely reducible. The covariant tensorial irreducible representations of $\mathfrak{sl}(m,n)$ were studied by Berele and Regev in \cite{BR}, through a generalization of the Schur-Weyl duality. Especially, they consider shapes $\lambda=(a_1,\dots,a_m,a'_1,\dots,a'_{n-1})$, placed in the hook, {\sl i.e.} built by adding under a shape having $a_j$ column with height $j$ a shape having $a'_i$ rows with length $i$, with the restriction $a_m\geq\sup\{i,~a'_i>0\}$. Choose the standard tableau $St_\lambda$, with shape $\lambda$, then $\mathbb S^\lambda$ is the space of $v\in V^{\otimes\sum ja_j+ia'_i}$ such that $vc_{St_\lambda}=v$. Moreover, a basis for $\mathbb S^\lambda$ is labelled by the semistandrad tableaux with shape $\lambda$ (see \cite{BR,KW}).\\

For a simple Lie algebra $\mathfrak g$, the reduced shape algebra $\mathbb S_{red}^\bullet$ is the quotient of the shape algebra by the ideal generated by $v_\lambda-1$, where $v_\lambda$ is a well chosen highest weight vector in $\mathbb S^\lambda$. Let $\mathfrak n^+$ be the nilpotent factor in the Iwasawa decomposition of $\mathfrak g$, as a $\mathfrak n^+$ module, $\mathbb S_{red}^\bullet$ is indecomposable, it is the union of the modules $\mathbb S^\lambda|_{\mathfrak n^+}$, with the stratification $\mathbb S^\mu|_{\mathfrak n^+}\subset\mathbb S^\lambda|_{\mathfrak n^+}$, if $\mu\leq \lambda$.\\

A way to describe a basis for $\mathbb S_{red}^\bullet$, respecting this stratification is the selection, among the semistandard tableaux of some tableaux called quasistandard. The {\sl diamond cone} is by definition the collection of all quasistandard tableaux. The projection map $p:\mathbb S^\bullet\longrightarrow\mathbb S_{red}^\bullet$ becomes a $push$ mapping on rows of the semistandard tableaux. This $push$ mapping can be defined through the use of the `jeu de taquin'. This program was completed for $\mathfrak{sl}(m)$, $\mathfrak{sp}(2n)$, $\mathfrak{so}(2n+1)$ and the rank 2 semisimple Lie algebras and $\mathfrak{sl}(m,1)$ (see \cite{AAB,AAK,AK,Kh}).\\

In this paper, we study this construction for the case of the Lie superalgebra $\mathfrak{sl}(m,n)$, using $push$ functions and `super jeu de taquin' corresponding to horizontal translations on the $m$ first rows and vertical translations on the bottom of the first $n-1$ columns.\\  

The paper is organized as follows. In Section 2, we recall the construction of the reduced shape algebra and quasistandard tableaux for $\mathfrak{sl}(m)$.

In Section 3, we recall the fundamental results on the Lie superalgebra $\mathfrak{sl}(m,n)$ and its finite dimensional irreducible representations.

In Section 4 and 5, we describe the shape algebra of $\mathfrak{sl}(m,n)$ and define the semistandard tableaux for $\mathfrak{sl}(m,n)$.

To describe the multiplication of the shape algebra directly on the tableaux, we need to study relations between tensor products of vectors in $\mathbb C^{m,n}$. These relations are the Garnir and Pl\"ucker relations, in Section 6, we recall their equivalence. These relations generate the ideal defining the shape algebra $\mathbb S^\bullet$ as a quotient of the tensor algebra $T(\mathbb C^{m,n})$.

With these relations, in Section 7, we define a multiplication $\star$ on the set of semistandard tableaux, which corresponds to the multiplication in the shape algebra.

In Section 8, we define the reduced shape algebra for $\mathfrak{sl}(m,n)$ and study the stratification of the corresponding $\mathfrak n^+$ module.

In Section 9, we define quasistandard tableaux, and the $push$ function. We prove that these tableaux form a basis for the reduced shape algebra, well adapted to the above stratification. 

Finally, in Section 10, we define the `super jeu de taquin' on skew semistandard tableaux and prove that the $push$ function is the result of a succession of actions of this `super jeu de taquin'.

This achieves the combinatorial description of the reduced shape algebra of $\mathfrak{sl}(m,n)$, {\sl i.e.} the description of the diamond cone for $\mathfrak{sl}(m,n)$.\\


\section{Diamond cone for $\mathfrak{sl}(m)$}

\subsection{Shape algebra of $\mathfrak{sl}(m)$ and semistandard tableaux }

\

Recall that the Lie algebra $\mathfrak{sl}(m) = \mathfrak{sl}(m,\mathbb C)$ is the set of $m\times m$ traceless matrices, it is the Lie algebra of the Lie group $SL(m)$ of $m\times m$ matrices, with determinant 1.\\

First, put $V=\mathbb C^m$, denote $(e_1,\dots,e_m)$ its canonical basis. Then $\mathfrak{sl}(m)$ acts naturally on $T(V)$. For any sequence $\lambda=(a_1,\dots,a_{m-1})$ of natural numbers, see $\lambda$ as the shape of a tableau having $a_{m-1}$ columns with height $m-1$, \dots, $a_1$ columns withheight $1$. Suppose that the heights of the columns in $\lambda$ are $c_1,\dots,c_k$, then to any Young tableau with shape $\lambda$, {\sl i.e.} to any filling $T$ of the shape $\lambda$ by entries $t_{i,j}$ in $\{1,\dots,m\}$, put:
$$
e^T=(e_{t_{1,1}}\otimes e_{t_{2,1}}\otimes\dots\otimes e_{t_{c_1,1}})\otimes(e_{t_{1,2}}\otimes\dots\otimes e_{t_{c_2,2}})\otimes\dots\otimes(e_{t_{1,k}}\otimes\dots\otimes e_{t_{c_k,k}}).
$$
For instance, if $S^0_\lambda$ is the tableau defined by $s^0_{i,j}=i$ (for any $i$ and $j$),
$$
e^{S^0_\lambda}=(e_1\otimes e_2\otimes\dots\otimes e_{c_1})\otimes(e_1\otimes\dots\otimes e_{c_2})\otimes\dots\otimes(e_1\otimes\dots\otimes e_{c_k}).
$$

We use the Young symmetrizer $c_\lambda$ associated to the standard tableau $St_\lambda$, with shape $\lambda$, defined by
$$
St_\lambda=(st_{i,j}),\qquad st_{i,j}=i+\sum_{j<i}c_j,
$$ 
foe any $i$ and $j$. For any Young tableau $T$, put $e_T=e^Tc_\lambda$. Then the polyvector:
$$
v_\lambda=e_{S^0_\lambda}=(e_1\wedge e_2\wedge\dots\wedge e_{m-1})^{a_{m-1}}\otimes\dots\otimes (e_1)^{a_1}\in Sym^{a_{m-1}}(\wedge^{m-1}V) \otimes\dots\otimes Sym^{a_1}(V).
$$
By construction, $v_\lambda$ is a primitive weight vector for the action of $\mathfrak{sl}(m)$ on $T(V)$, $v_\lambda c_\lambda=v_\lambda$, thus it is a highest weight vector for $\mathbb S^\lambda$, and $\mathbb S^\lambda$ is the submodule in $Sym^\bullet(\wedge V)$ generated by $v_\lambda$.

\begin{defn}

\

A Young tableaux of shape $\lambda$ is semistandard if its entries are increasing along each row and strictly increasing along each column. We note $SS^\lambda$ the set of all semistandard tableaux with shape $\lambda$.\\
\end{defn}

This notion allows to select a basis for $\mathbb S^\lambda$.

\begin{thm}

\

If $\lambda=(a_1,\dots,a_{m-1})$ is a shape, a basis for ${\mathbb S}^\lambda$ is given by the family of polyvectors $(e_T)$, where $T$ is a semistandard Young tableau with shape $\lambda$.\\
\end{thm}

On the other hand, for any tableau $C$ with one column: $C=(c_{1,j})=(k_j)$, the polyvector $e_C$ is:
$$
e_C=e_{k_1}\wedge\dots\wedge e_{k_c}. 
$$
It is possible to realize the space $\mathbb S^\bullet=\bigoplus\mathbb S^\lambda$ as a quotient of the algebra $Sym^\bullet(\wedge V)$ by an ideal.

\begin{thm}\label{shapeslm}

\

For any column
$$
C=\begin{array}{|c|}
\hline
i_1\\
\hline
\vdots\\
\hline
i_p\\
\hline
\end{array},
$$
put $e_C=e_{\{i_1,\dots,i_p\}}$.

Then the space ${\mathbb S}^\bullet=\bigoplus_{\lambda}~~{\mathbb S}^\lambda,$ is isomorphic to the quotient of the algebra $Sym^{\bullet}(\bigwedge V)$ by the ideal $\mathcal P$ generated by the Pl\"ucker relations: for any $p\geq q$, for any column $C$ with height $p$ and column $D$ with height $q$,
$$
e_{\{i_1,i_2,\dots,i_p\}}e_{\{j_1,j_2,\dots,j_q\}}+\sum\limits_{\begin{smallmatrix}
A \subset\{i_1,\dots,i_p\}\\
\# A = r\leq q
\end{smallmatrix}}\pm e_{(\{i_1,\dots,i_p\}\setminus A)\cup\{j_1,\dots,j_r\}}e_{A\cup \{j_{r+1},\dots,j_q\}}=0.
$$
\end{thm}

This defines a structure of associative and commutative algebra on the space $\mathbb S^\bullet$. We call this space the shape algebra of $\mathfrak{sl}(m)$.\\


\subsection{Diamond cone for $\mathfrak{sl}(m)$ and quasistandard tableaux}


\

Let us now put:

\begin{defn}

\

We call reduced shape algebra or diamond cone the quotient:
$$
\mathbb{S}^\bullet_{red}=\mathbb{S}^\bullet/\big<v_\lambda-1~~ \big >.
$$
\end{defn}

\

This algebra is a natural $\mathfrak n^+$ module, where $\mathfrak n^+$ is the Lie algebra of strictly upper triangular $m\times m$ matrices. This module is indecomposable, if $p$ is the canonical projection map from $\mathbb S^\bullet$ to $\mathbb S_{red}^\bullet$, the restriction of $p$ to any $\mathbb S^\lambda$ is an isomorphism of $\mathfrak n^+$ module but we have $p(\mathbb S^\mu)\subset p(\mathbb S^\lambda)$ if $\mu\leq\lambda$, {\sl i.e.} if $\lambda=(a_1,\dots,a_{m-1})$ and $\mu=(b_1,\dots,b_{m-1})$, $b_k\leq a_k$ for any $k$ (see \cite{ABW}).\\

In order to build a basis for $\mathbb{S}^\bullet_{red}$, well adapted to the above stratification, a notion of quasistandard Young tableau for $\mathfrak{sl}(m)$ was introduced. Let us first recall this notion in a presentation usefull for the next sections.\\

Let $\nu=(c_1,\dots,c_h,0,\dots,0)$ be a shape. We call trivial tableau with shape $\nu$ for $\mathfrak{sl}(m)$ and denote $trivial(\nu)$ the tableau with shape $\nu$ and entries $tr_{i,j}\in\{1,2,\dots,m\}$ such that:
$$
tr_{i,j}=i,\qquad\text{ for all }\quad i=1,\dots,h,~~j=1,\dots,\ell_i=\sum_{k\geq i}c_k.
$$

That means we fill up the shape $\nu$, row by row, with the number of the row.\\

Let $T$ be a semistandard Young tableau, with shape $\lambda$ and entries $t_{i,j}$.

If $\nu\leq \lambda$, we say that the trivial tableau with shape $\nu$ is extractable from $T$ if
\begin{itemize}
\item[1.] The subtableau, with shape $\nu$, placed on the top and left in $T$ is trivial, or:
$$
t_{i,j}=i,\qquad\text{ for all }\quad i=1,\dots,h,~~j=1,\dots,\ell_i.
$$
\item[2.] It is possible to extract this subtableau from $T$ by pushing each row of $T$, with number $i\leq h$, $\ell_i$ steps from the right to the left, getting a new semistandard tableau or (with our convention): for all $i$, and all $j$,
$$
t_{i,j+\sum_{k\geq i}c_k}<t_{i+1,j+\sum_{k\geq i+1}c_k}.
$$
This is equivalent to
\item[2'.] For all $i$, and all $j\geq\sum_{k>i}c_k$,
$$
t_{i,j+c_i}<t_{i+1,j}.
$$
\end{itemize}

\begin{defn}\label{qasistandardsl(m)}

\

Let $T$ be a semistandard tableau with shape $\lambda$ ($T\in SS^\lambda$). If there is no trivial extractable tableau in $T$ (except the empty tableau with shape $(0,_dots,0)$), we say that $T$ is a quasistandard tableau. Denote $QS^\lambda$ the set of all quasistandard tableaux with shape $\lambda$.\\
\end{defn}

It is shown in \cite{ABW} that for any $T$ in $SS^\lambda$, there is a greatest trivial extractable subtableau $S$ in $T$, with shape $\nu\leq\lambda$. We get thus, as above, a new semistandard tableau, denoted $push(T)$, by extracting $S$. In fact, $push(T)$ is quasistandard, with shape $\mu=\lambda-\nu$.

For instance, if $m=4$, the tableau
$$
T=\begin{array}{l}
\begin{array}{|c|c|c|c|c|}
\hline
1&1&2&2&3\\
\hline
\end{array}\\
\begin{array}{|c|c|c|}
2&3&4\\
\hline
\end{array}\\
\begin{array}{|c|}
4\\
\hline
\end{array}\\
\end{array}
$$
is in $SS^{(2,2,1)}$ and
$$
push(T)=push\left(\begin{array}{l}
\begin{array}{|c|c|c|c|c|}
\hline
1&1&2&2&3\\
\hline
\end{array}\\
\begin{array}{|c|c|c|}
2&3&4\\
\hline
\end{array}\\
\begin{array}{|c|}
4\\
\hline
\end{array}\\
\end{array}\right)=\begin{array}{l}
\begin{array}{|c|c|c|}
\hline
2&2&3\\
\hline
\end{array}\\
\begin{array}{|c|c|}
3&4\\
\hline
\end{array}\\
\begin{array}{|c|}
4\\
\hline
\end{array}\\
\end{array}
$$
is in $QS^{(1,1,1)}$.\\

Morover, the map $push$ is bijective from $SS^\lambda$ onto the disjoint union $\bigsqcup_{\mu\leq \lambda}QS^\mu$. And it is proved in \cite{ABW} that

\begin{thm}

\

The family $p(e_T)$ for $T$ quasistandard is a basis for the reduced shape algebra $\mathbb S_{red}^\bullet$ of $\mathfrak{sl}(m)$ well adapted to the stratification $p(\mathbb S^\mu)\subset p(\mathbb S^\lambda)$ if $\mu\leq\lambda$.
\end{thm}

We call this basis the diamond cone for $\mathfrak{sl}(m)$.\\

\begin{exple}

\

For $m=3$, we get the following picture of a part of the diamond cone:
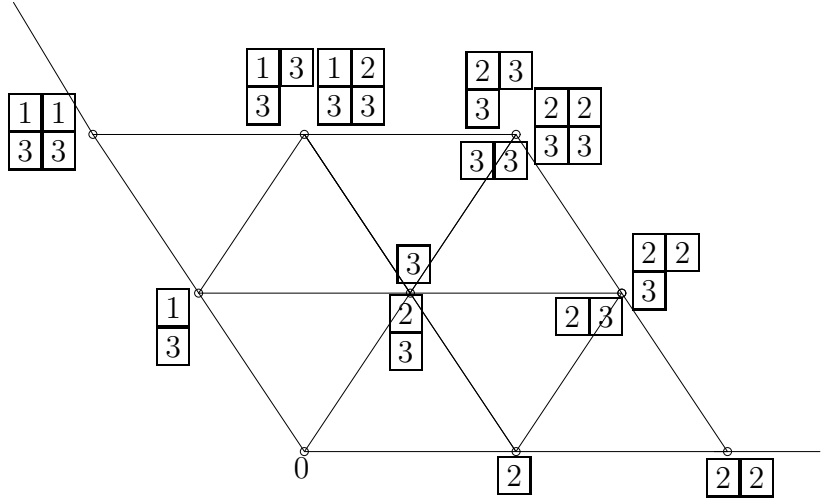
\eject

\

\vskip5cm
\begin{center}
\begin{picture}(180,80)(200,20)

\path(175,150)(335,150) \path(335,150)(415,30) \path(255,30)(415,30)
\path(175,150)(255,30) \path(215,90)(375,90) \path(335,30)(255,150)
\path(295,90)(255,150) \path(375,90)(335,30) \path(295,90)(335,30)
\path(295,90)(335,150) \path(295,90)(335,150) \path(255,150)(295,90)
\path(295,90)(255,30) \path(255,150)(215,90) \path(145,200)(175,150)
\path(410,30)(450,30)

\put(175,150){\circle{3}} \put(335,150){\circle{3}}
\put(415,30){\circle{3}} \put(255,30){\circle{3}}
\put(215,90){\circle{3}} \put(375,90){\circle{3}}
\put(335,30){\circle{3}} \put(295,90){\circle{3}}
\put(255,150){\circle{3}} \put(375,90){\circle{3}}

\put(194,74) {\renewcommand{\arraystretch}{0.7}$\begin{array}{l}
\framebox{$1$}\\
\framebox{$3$}\\
\end{array}$}

\put(282,72){\renewcommand{\arraystretch}{0.7}$\begin{array}{l}
\framebox{$2$}\\
\framebox{$3$}\\
\end{array}$}

\put(228,165){\renewcommand{\arraystretch}{0.7}$\begin{array}{l}
\framebox{$1$}\framebox{$3$}\\
\framebox{$3$}\\
\end{array}$}

\put(255,165){\renewcommand{\arraystretch}{0.7}$\begin{array}{l}
\framebox{$1$}\framebox{$2$}\\
\framebox{$3$}\framebox{$3$}\\
\end{array}$}

\put(138,148){\renewcommand{\arraystretch}{0.7}$\begin{array}{l}
\framebox{$1$}\framebox{$1$}\\
\framebox{$3$}\framebox{$3$}\\
\end{array}$}

\put(337,150){\renewcommand{\arraystretch}{0.7}$\begin{array}{l}
\framebox{$2$}\framebox{$2$}\\
\framebox{$3$}\framebox{$3$}\\
\end{array}$}

\put(311,164){\renewcommand{\arraystretch}{0.7}$\begin{array}{l}
\framebox{$2$}\framebox{$3$}\\
\framebox{$3$}\\
\end{array}$}

\put(309,137){\renewcommand{\arraystretch}{0.7}$\begin{array}{l}
\framebox{$3$}\framebox{$3$}\\
\end{array}$}

\put(328,17){\framebox{$2$}}

\put(290,97){\framebox{$3$}}

\put(402,17){\renewcommand{\arraystretch}{0.7}$\begin{array}{l}
\framebox{$2$}\framebox{$2$}\\
\end{array}$}

\put(374,95){\renewcommand{\arraystretch}{0.7}$\begin{array}{l}
\framebox{$2$}\framebox{$2$}\\
\framebox{$3$}\\
\end{array}$}

\put(345,78){\renewcommand{\arraystretch}{0.7}$\begin{array}{l}
\framebox{$2$}\framebox{$3$}\\
\end{array}$}

\put(251,20){$0$}

\end{picture}
\end{center}
\end{exple}
\vskip1cm


\section{The special linear Lie superalgebra $\mathfrak{sl}(m,n)$}


\
Let us recall some basic definitions from Lie superalgebra theory (\cite{V}).\\

A Lie superalgebra $\mathfrak g=\mathfrak g_{\overline{0}} \oplus\mathfrak g_{\overline{1}}$ is a $\mathbb{Z}_2$-graded algebra with product $[.,.]$, {\sl i.e.} if $a$ is in $\mathfrak g _\alpha$, $b$ in $\mathfrak g_\beta$,($\alpha, \beta \in \mathbb{Z}_2=\{\overline{0},\overline{1}\}$)
then $[a,b]$ is in $\mathfrak g_{\alpha+\beta}$. The bracket satisfies the following axioms:
$$
[a,b]= -(-1)^{\alpha \beta} [b,a];\qquad\qquad~~ [a,[b,c]]=[[a,b],c]+ (-1)^{\alpha \beta}[b,[a,c]],~~\forall c \in \mathfrak g.
$$

In this paper, we are interested on the case $\mathfrak g=\mathfrak{sl}(m,n)$, with $m,n \in \mathbb{N}$. Recall that $\mathfrak{sl}(m,n)$ can be viewed as the set of all $(m+n)^2$ matrices $X=\left(\begin{array}{cc}A&B\\C&D\\
\end{array}\right)$ over the complex number field $\mathbb{C}$ where $A$ is a ${m \times m}$ matrix, $B$ a ${m \times n}$ one, $C$ a ${n \times m}$ one, $D$ a ${n \times n}$ one, and $str(X)=trA-trD=0$. The even part $\mathfrak g_{\overline{0}}$ of $\mathfrak g$ consists of matrices of the form $\left(\begin{array}{cc}A&0\\0&D\\
\end{array}\right)$, the odd part $\mathfrak g_{\overline{1}}$ of $\mathfrak g$ consists of matrices of the form $\left(\begin{array}{cc}0&B\\C&0\\
\end{array}\right)$ and $\mathfrak g_{\overline{0}}\cong \mathfrak{sl}(m)\oplus \mathfrak{sl}(n)\oplus\mathbb{C}$. 

It is known (see \cite{K}) that the center $\mathfrak z$ of $\mathfrak{sl}(m,n)$ is trivial if $m\neq n$, one dimensional if $m=n$, and the quotient $\mathfrak{sl}(m,n)/\mathfrak z$ is always simple.

Let
$$
\mathfrak h=\{diag(h_1,\ldots,h_{m+n}),\quad h_s \in \mathbb{C},\quad 1 \leq s \leq m+n,\quad\text{and } \displaystyle\sum_{i=1}^{m}h_i=\displaystyle\sum_{j=1}^{n}h_{m+j}\}.
$$
Then $\mathfrak h$ is a Cartan subalgebra of $\mathfrak g$. The corresponding root system will be denoted by $\Delta$. Put $\epsilon_s(diag(h_1,\ldots,h_{m+n})=h_s$. The roots can be expressed in terms of linear functions $\epsilon_1$,\ldots,$\epsilon_m$, $\delta_1=\epsilon_{m+1}$,\ldots, $\delta_n=\epsilon_{m+n}$. Let $\Delta_{\overline{0}}$ be the set of even roots, let $\Delta_{\overline{1}}$ be the set of odd
roots, then
$$
\Delta_{\overline{0}}=\{\epsilon_i-\epsilon_j, \delta_i-\delta_j,  i \neq j \}~~\hbox{and}~~
\Delta_{\overline{1}}=\{\pm (\epsilon_i-\delta_j)\}.
$$
We choose
$$
\Pi=\{\epsilon_1-\epsilon_2,\epsilon_2-\epsilon_3,\ldots,\epsilon_m-\delta_1,\delta_1-\delta_2,\ldots,\delta_{n-1}-\delta_n\}
$$
as a simple root system.\\

The corresponding nilpotent factor in $\mathfrak{sl}(m,n)$ is:
$$
\mathfrak n^+=\left\{X=\left(\begin{matrix}A&B\\0&D\end{matrix}\right),\quad A,D\text{ strictly upper triangular }\right\}.
$$

The weight space $\mathfrak h^\star$ is spanned by $\epsilon_i$  and $\delta_j$. So, a weight $\lambda \in \mathfrak h^\star$ can be written as
$$\lambda=\sum_{i=1}^m
\lambda_i \epsilon_i + \sum_{j=1}^n \mu_j\delta_j~\hbox{ with }~ \sum_{i=1}^m
\lambda_i -\sum_{j=1}^n \mu_j=0.
$$
Following \cite{V}, we put $a_i=\lambda_i-\lambda_{i+1}$ for $i<m$, $a_m=\lambda_m+\mu_n$ and $a'_j=\mu_j-\mu_{j+1}$ for $j<n$. A weight $\lambda$ is called integral dominant if $a_i \in \mathbb{N}$ for $i\neq m$, whereas $a_m\in \mathbb{Z}$. The set of integral dominant weights is denoted by $\Lambda$.\\

We just denote elements in $\Lambda$ by $\lambda=(a,a')=((a_1,\dots,a_m),(a'_1,\dots,a'_{n-1}))$. Remark the ordering on $\Lambda$ is defined by $\mu=(b,b')\leq\lambda=(a,a')$ if and only if $\lambda-\mu$ is dominant, if and only if $b_i\leq a_i$ for $i<m$ and $b'_j\leq a'_j$ for $j<n$.\\ 

Let $\lambda$ be dominant. From the theory of reductive Lie algebras, it follows that there exists a unique finite-dimensional irreducible
$\mathfrak{sl}(m,n)_{\overline{0}}$ module $V_0(\lambda)$ with highest weight $\lambda$. Let $v_\lambda$ be a highest weight vector for
$V_0(\lambda)$.\\

\begin{defn}

\

For any $\lambda \in \Lambda$, the Kac module $\overline{V}(\lambda)$ is the induced module
$$
\overline{V}(\lambda)=Ind^{\mathfrak{sl}(m,n)}_{\mathfrak{sl}(m,n)_{\overline{0}}~
\oplus \mathfrak n^+_{\overline{1}}} V_0(\lambda)=
U(\mathfrak{sl}(m,n)) \otimes _{\mathfrak{sl}(m,n)_{\overline{0}}
~\oplus \mathfrak n^+_{\overline{1}}}{ V_0(\lambda)},
$$
where $U(\mathfrak{sl}(m,n))$ is the universal enveloping superalgebra of $\mathfrak{sl}(m,n)$.\\
\end{defn}

$\overline{V}(\lambda)$ is a finite-dimensional $\mathfrak{sl}(m,n)$ module of dimension $2^{mn} \times dim \big(V_0(\lambda)\big)$. Unfortunately, $\overline{V}(\lambda)$ is not always an irreducible module. Since $\overline{V}(\lambda)$ is a
highest weight module, it contains a unique maximal submodule $M(\lambda)$ such that the quotient module:
$$
\overline{V}(\lambda)/ M(\lambda)
$$
is a finite dimensional simple $\mathfrak{sl}(m,n)$ module with highest weight $\lambda$.\\

\begin{defn}

\ 

For any $\lambda$ in $\Lambda$, we denote the unique simple $\mathfrak{sl}(m,n)$-module with highest weight $\lambda$ by
$$
V(\lambda)=\mathbb S^{\lambda}=\overline{V}(\lambda)/ M(\lambda).
$$

\end{defn}
Kac (\cite{K}) proved the following result:

\begin{thm}

\

Every finite-dimensional simple $\mathfrak{sl}(m,n)$ module is
isomorphic to a module of the type
$V(\lambda)$ where $\lambda$ is
an integral dominant weight.

Moreover, any finite-dimensional irreducible $\mathfrak{sl}(m,n)$ module is uniquely characterized by its integral dominant weight $\lambda$.\\
\end{thm}

Put $\rho=\frac{1}{2}\sum_{\alpha \in\Delta^{+}_{\overline{0}}} \alpha- \frac{1}{2}\sum_{\beta \in
\Delta^{+}_{\overline{1}}} \beta$, where $\Delta^{+}_{\overline{0}}$ (resp. $\Delta^{+}_{\overline{1}}$)is the set of even (resp. odd) roots.
Or explicitly in the $\epsilon\delta$-basis
$$
\rho= \frac{1}{2}\displaystyle\sum_{i=1}^m
(m-n-2i+1)\epsilon_i+\frac{1}{2}\displaystyle\sum_{j=1}^n
(m+n-2j+1)\delta_j.
$$
There is a symmetric form $(~,~)$ on $\mathfrak h^\star$ induced by the invariant symmetric form $(X,Y)\mapsto str(XY)$ on $\mathfrak{sl}(m,n)$, and in the natural basis it is the restriction to $\mathfrak h^\star$ of the form $(\epsilon_i,\epsilon_j)=\delta_{ij}$, $(\epsilon_i,\delta_j)=0$ and $(\delta_i,\delta_j)=-\delta_{ij}$, where $\delta_{ij}$ is the usual Kronecker symbol.\\

If $\lambda$ is a dominant weight of $\mathfrak{sl}(m,n)$ then $\lambda$ is said:
\begin{itemize}
\item[i)] typical if $(\lambda+\rho,\beta) \neq 0$,~~~~for all $\beta \in \Delta^+_{\overline{1}}$,\\
\item[ii)] atypical if there exist $\beta \in \Delta^+_{\overline{1}}$ such that $(\lambda+\rho,\beta) = 0$.
\end{itemize}

\begin{thm}

\

Let $\lambda$ be a dominant weight. The Kac module $\overline{V}(\lambda)$ is an irreducible $\mathfrak{sl}(m,n)$ module if and only if its highest weight $\lambda$ is typical.

In this case, we call $V(\lambda)=\overline{V}(\lambda)$ a typical module, otherwise $V(\lambda)\neq \overline{V}(\lambda)$ is called atypical module.\\
\end{thm}

In this paper we are interested by the modules appearing in the tensor algebra $T(V)=\bigoplus_{N\geq0} V^{\otimes N}$, where $V=\mathbb C^{m,n}$ is the natural $\mathfrak{sl}(m,n)$ module. Let us recall the $\mathfrak{sl}(m,n)$ (left) action on $V^{\otimes N}$ is:
$$
X.(v_1\otimes\ldots\otimes v_N)=\sum_{i=1}^N(-1)^{|X|(|v_1|+\ldots+|v_{i-1}|)}(v_1\otimes\ldots\otimes Xv_i\otimes\dots\otimes v_N).
$$
We call these modules covariant tensorial modules, they were studied by \cite{BR,KW,S}. On the other hand, the group $S_N$ acts on the right by:
$$
(v_1\otimes\dots\otimes v_N)\cdot \sigma=\varepsilon_{|v|}(\sigma)(v_{\sigma(1)}\otimes\ldots\otimes v_{\sigma(N)}),
$$
where $\varepsilon_{|v|}(\sigma)$ is the sign of the permutation induced by $\sigma$ on the odd $v_i$.

These two actions commute : $X.(v\cdot\sigma)=(X.v)\cdot\sigma$. Berele and Regev prove:

\begin{thm}

\

All the covariant tensorial modules are completely reducible, among them, the simple modules are exactly the modules $V(\lambda)=\mathbb S^{(a,a')}$, with $a_m$ an integral positive number and $\tilde{a}_m=a_m-\sup\{j,~a'_j>0\}\geq0$.

We note $\Lambda_{cov}$ the set of such covariant dominant integral weights.\\
\end{thm}

Denote $e_1,\dots,e_{m+n}$ the canonical basis for $V$, where $e_1,\dots, e_m$ are even ($|e_i|=0$), and $e_{m+1},\dots,e_{m+n}$ are odd ($|e_{m+j}|=1$). If $x,~y$ are homogeneous vectors in $V$, we put:
$$
x\cdot y=\frac{1}{2}(x\otimes y+(-1)^{|x||y|}y\otimes x)\qquad x\wedge y=\frac{1}{2}(x\otimes y-(-1)^{|x||y|}y\otimes x).
$$

Put $A'_j=a'_j+\dots+a'_{n-1}$, define $\{s_1<\dots<s_p\}$ as the set of $j$ sucht that $a'_j>0$ ($s_p=j_0$), define $s_0=0$ and the numbers $\tilde{a}_k$ as follows:
$$
\tilde{a}_k=0 ~\text{ if }~k\notin\{A'_{s_1}>\dots >A'_{s_p}\}~\text{ and }~\tilde{a}_{A'_{s_r}}=s_r-s_{r-1}.
$$
We choose the highest weight vector $v_\lambda$ ($\lambda\in\Lambda_{cov}$) in the space:
$$\aligned
&Sym^{\tilde{a}_1}(\wedge^{m+A'_1}V)\otimes\dots\otimes Sym^{\tilde{a}_{n-1}}(\wedge^{m+A'_{n-1}}V)\otimes Sym^{\tilde{a}_m}(\wedge^mV)\otimes\\
&\otimes Sym^{a_{m-1}}(\wedge^{m-1}V)\otimes\dots\otimes Sym^{a_1}(V).
\endaligned
$$

This choice is associated to the choice of a Young symmetrizer $c_\lambda$. We start with the vector:
$$\aligned
w&=(e_1\otimes\dots\otimes e_m\otimes e_{m+1}\otimes\dots\otimes e_{m+1})\otimes(e_1\otimes\dots\otimes e_m\otimes e_{m+2}\otimes\dots\otimes e_{m+2})\otimes\dots\\
&\hskip1cm \dots\otimes(e_1\otimes\dots\otimes e_m\otimes e_{m+n-1}\otimes\dots\otimes e_{m+n-1})\otimes(e_1\otimes\dots\otimes e_m)\otimes\dots\\
&\hskip1cm \dots\otimes(e_1)
\endaligned
$$
where the first $\tilde{a}_1$ parenthesis contain $A'_1$ vector $e_{m+i}$, the $\tilde{a}_2$ following parenthesis contain $A'_2$ vector $e_{m+i}$, and so on, there are $a_i$ parenthesis of the form $e_1\otimes\ldots\otimes e_i$. Then we define $v_\lambda$ as a multiple of $c_\lambda(w)$, for a good choice of $c_\lambda$, we get explicitely:

\begin{prop} (\cite{BR})\label{vlambda}

\

We keep all the preceding notations, and denote $G$ the group $S_1^{a'_1}\times S_2^{a'_2}\times\dots\times S_{n-1}^{a'_{n-1}}$. Then the simple module $V(\lambda)=\mathbb S^{(a,a')}$ is realized as the submodule in $T(V)$ generated by its highest weight vector:
$$\aligned
v_\lambda&=\alpha_\lambda\Big[\sum_{(\tau_{A'_1},\tau_{A'_1-1},\dots,\tau_1)\in G}\prod_{k=1}^{A'_1}\varepsilon(\tau_k)\left(e_1\wedge\dots\wedge e_m\wedge e_{m+\tau_1(1)}\wedge\dots\wedge e_{m+\tau_{A'_1}(1)}\right)\otimes\\
&\hskip1cm\otimes\left(e_1\wedge\dots\wedge e_m\wedge e_{m+\tau_1(2)}\wedge\dots\wedge e_{m+\tau_{A'_2}(2)}\right)\otimes\\
&\hskip1cm\ldots\otimes\left(e_1\wedge\dots\wedge e_m\wedge e_{m+\tau_1(n-1)}\wedge e_{m+\tau_{A'_{n-1}}(n-1)}\right)\Big]\otimes\\
&\hskip1cm\otimes\left(e_1\wedge\dots\wedge e_m\right)^{a_m-j_0}\otimes\left(e_1\wedge\dots\wedge e_{m-1}\right)^{a_{m-1}}\otimes\dots\otimes(e_1)^{a_1},
\endaligned
$$ 
where the numerical factor $\alpha_\lambda$ is:
$$
\alpha_\lambda=\prod_{i=1}^m A_i!\prod_{j=1}^{n-1}(m+A'_j)!(m!)^{a_m-j_0}\prod_{k=1}^{m-1}\left((m-k)!\right)^{a_k}.
$$
\end{prop}

The value of the coefficient $\alpha_\lambda$ will be usefull in Section \ref{semistand}.

With the notation $A_i=a_i+\dots+a_m$, the highest weight $\lambda$ can be written, modulo the supertrace as:
$$
\lambda=\sum_{i=1}^m A_i\epsilon_i+\sum_{j=1}^{n-1}A'_j\delta_j.
$$


\section{Shape algebra}\label{shape}


\

As for the Lie algebra $\mathfrak{sl}(m)$, knowing the simple covariant tensorial modules $\mathbb S^\lambda$ for $\mathfrak{sl}(m,n)$, we can define a structure of commutative algebra on their direct sum. More precisely, we define the space:
$$
\mathbb S^\bullet=\bigoplus_{\lambda\in\Lambda_{cov}}\mathbb S^\lambda,
$$

To define the product on $\mathbb S^\bullet$, we need to consider its dual $(\mathbb S^\bullet)^\vee$ as a natural $\mathfrak{sl}(n,m)$-module.

First each $\mathbb S^\lambda$ is a finite dimensional $\mathfrak{gl}(m)\times\mathfrak{gl}(n)$-module, thus the corresponding group acts on $\mathbb S^\lambda$, we consider the element:
$$
w=\left(\begin{matrix}\left(\begin{matrix}0&&&&1\\&&&\cdot&\\&&\cdot&&\\&\cdot&&&\\1&&&&&0\end{matrix}\right)&0\\0&\left(\begin{matrix}0&&&&1\\&&&\cdot&\\&&\cdot&&\\&\cdot&&&\\1&&&&&0\end{matrix}\right)\end{matrix}\right)
$$
in $GL(m)\times GL(n)$. Put $\tilde{e}_i=e_{m+1-i}$, $\tilde{e}_{m+j}=e_{m+n+1-j}$. In $\mathbb S^\lambda$, there is the vector $\tilde{v}_\lambda=w\cdot v_\lambda$, which have the same expression as $v_\lambda$, but with the $\tilde{e}_k$ replacing the $e_k$. Its weight is:
$$
w(\lambda)=A'_{n-1}\delta_2+A'_{n-2}\delta_3+\dots+A'_1\delta_n+A_m\epsilon_1+A_{m-1}\epsilon_2+\dots+A_1\epsilon_m.
$$

Now the contragredient module $(\mathbb S^\lambda)^\vee$, defined by $(X,f)\mapsto -^tX\cdot f$, where $f$ belongs to the dual $(S^\lambda)^\star$ and $^tX\cdot$ is the supertranspose of the transformation $X\cdot$ on $\mathbb S^\lambda$, contains a vector with weight $-w(\lambda)$.

We consider this $\mathfrak{sl}(m,n)$-module as a $\mathfrak{sl}(n,m)$-module, by the identification of these two Lie superalgebras through:
$$
\left(\begin{matrix}A&B\\ C&D\end{matrix}\right)\longmapsto \left(\begin{matrix}D&C\\ B&A\end{matrix}\right),
$$
then the system of simple roots for $\mathfrak{sl}(n,m)$ is 
$$
\Pi^\vee=\{\delta_1-\delta_2,\ldots,\delta_{n-1}-\delta_n,\delta_n-\epsilon_1,\epsilon_1-\epsilon_2,\epsilon_2-\epsilon_3,\ldots, \epsilon_{m-1}-\epsilon_m\}.
$$
With this choice, $\lambda^\vee=-w(\lambda)$ becomes an integral dominant weight for $\mathfrak{sl}(n,m)$, since it can be written $\lambda^\vee=(a^\vee,a^{'\vee})$, where:
$$
a^\vee_j=a'_{n-j}\quad(1\leq j\leq n-1),\quad a^\vee_n=a_m-j_0,\quad a^{'\vee}_i=a_{m-i}\quad(1\leq i\leq m-1).
$$
We shall just write
$$
(\mathbb S^\lambda)^\vee=(\mathbb S_{m,n}^\lambda)^\vee\simeq\mathbb S_{n,m}^{\lambda^\vee}=\mathbb S^{\lambda^\vee}.
$$

Now, for each $\mu$, $\nu$ in $\Lambda_{cov}$ such that $\mu+\nu=\lambda$, since $\lambda\mapsto\lambda^\vee$ is linear, there is a natural morphism of $\mathfrak{sl}(n,m)$-modules:
$$
\Delta:(\mathbb S^\lambda)^\vee\longrightarrow (\mathbb S^\mu)^\vee\otimes (\mathbb S^\nu)^\vee.
$$
By transposition, we get a natural morphism of $\mathfrak{sl}(m,n)$-modules:
$$
\star:\mathbb S^\mu\otimes \mathbb S^\nu\longrightarrow \mathbb S^\lambda.
$$

\begin{prop}

\

The so defined product $\star$, on $\mathbb S^\bullet$, is commutative and associative, and satisfies:
\begin{itemize}
\item[1.] $\star$ is an intertwining operator from $\mathbb S^\lambda\otimes\mathbb S^\mu$ to $\mathbb S^{\lambda+\mu}$,
\item[2.] when we fix the highest weight vector $v_\lambda$ in $\mathbb S^\lambda$ as in the preceding section, $v_\lambda\star v_\mu=v_{\lambda+\mu}$.\\
\end{itemize}
\end{prop}

Remark that $\mathbb S^\lambda\otimes\mathbb S^\mu$ being a submodule in $T(V)$, is completely reducible into a sum of simple modules isomorphic to some $\mathbb S^\nu$, with $\nu\leq \lambda+\mu$. However, the isotypic component corresponding to $\nu=\lambda+\mu$ is simple, generated by $v_\lambda\otimes v_\mu$. As a consequence, the two properties of this proposition characterize completely the multiplication $\star$.\\

\begin{defn}

\

The algebra $(\mathbb S^\bullet,\star)$ is called the shape algebra of $\mathfrak{sl}(m,n)$.\\
\end{defn}

Especially, for any $\lambda$, there is an injective morphism of $\mathfrak n^+$-module from $\mathbb S^\mu$ into $\mathbb S^{\lambda+\mu}$, given by:
$$
v\mapsto v_\lambda\star v.
$$

Conversely, suppose $\lambda$, $\mu$ are in $\Lambda_{cov}$, with $\mu\leq\lambda$.

If $\lambda-\mu$ is in $\Lambda_{cov}$, then $v\mapsto v_{\lambda-\mu}\star v$ is an injective morphism of $\mathfrak n^+$-modules from $\mathbb S^\mu$ into $\mathbb S^\lambda$.

Unfortunately, we can have $\lambda-\mu\notin\Lambda_{cov}$. Indeed if $\lambda=(a,a')$, and $\mu=(b,b')$, $\lambda-\mu=(a-b,a'-b')$ is integral dominant but we can have $a_m-b_m<\sup\{j,~a'_j>b'_j\}$. In this situation, we can define another injective map as follows.

Consider the weight $\eta=\sum_{i=1}^m\epsilon_i$. This weight is in $\Lambda_{cov}$, since $\eta=((0,\dots,1),0)$. There is a minimal natural number $k$ such that $\lambda-\mu+k\eta$ belongs to $\Lambda_{cov}$, namely:
$$
k=\sup(0,\sup\{j,~a'_j>b'_j\}-a_m+b_m).
$$

Then, as $\mathfrak n^+$-module, $\mathbb S^\lambda$ is isomorphic to $v_{k\eta}\star\mathbb S^\lambda$, and this module contains a $\mathfrak n^+$-module isomorphic to $\mathbb S^\mu$, namely $v_{\lambda-\mu+k\eta}\star\mathbb S^\mu$.\\

In the next section, we shall describe a linear basis for the shape algebra, given by semistandard Young tableaux.\\


\section{Semistandard tableaux for $\mathfrak{sl}(m,n)$}\label{semistand}


\subsection{Young tableaux}


\

A Ferrer diagram $F$ for $\mathfrak{gl}(m)$, is a juxtaposition of columns of empty boxes, the diagram is characterized by its shape $a=(a_1,\dots,a_m)$ where $a_i$ is the number of columns with height $i$.\\

The transpose of $F$ is the diagram $F^t$ having, for any $j$, $a_j$ rows with length $j$.\\

To build a diagram for $\mathfrak{sl}(m,n)$ (with $n\geq1$) we consider two Ferrer diagrams, one $F^+$ for $\mathfrak{gl}(m)$, with shape $a=(a_1,\dots,a_m)$ and one $F'$ for $\mathfrak{gl}(n-1)$ with shape $a'=(a'_1,\dots,a'_{n-1})$, we suppose the condition:
$$
a_m\geq \sup\{j,~a'_j\neq0\}\eqno{(*)}
$$
holds. Then, we first put (at the top) the diagram $F^+$, then a line at the separation between rows $m$ and $m+1$, finally the transpose $F^-$ of the diagram $F'$ under this line and at the left. We get a diagram $F=F^+\uplus F^-$ with shape:
$$
\lambda=(a,a')=((a_1,\dots,a_m),(a'_1,\dots,a'_{n-1})).
$$

Here is an example, for $\mathfrak{sl}(2,3)$: start with $F^+$ and $F'$:
$$
F^+=\begin{tabular}{|c|c|c|}
\hline
~~&~~&~~\\
\hline
~~&\multicolumn{1}{|c|}{~~}\\
\cline{1-2}
\end{tabular}\qquad
F'=\begin{tabular}{|c|c|c|c|c|}
\hline
~~&~~&~~&~~&~~\\
\hline
~~&~~&\\
\cline{1-3}
\end{tabular}
$$
Then the respective shapes are $a=(1,2)$, $a'=(2,3)$, the transpose of $F'$ is:
$$
F^-=(F')^t=\begin{tabular}{|c|c|}
\hline
~~&~~\\
\hline
~~&~~\\
\hline
~~&~~\\
\hline
\multicolumn{1}{|c|}{~~}\\
\cline{1-1}
\multicolumn{1}{|c|}{~~}\\
\cline{1-1}
\end{tabular}
$$
and the final Ferrer diagram is:
$$
F=F^+\uplus F^-=\begin{tabular}{c|c|c|c|c}
\cline{2-4}
\hskip 3cm&~~&~~&~~&\hskip 3cm\\
\cline{2-4}
&~~&~~&\multicolumn{1}{c}{~~}&\multicolumn{1}{c}{~~~~~~$_m$}\\
\hline
&~~&~~\\
\cline{2-3}
&~~&~~\\
\cline{2-3}
&~~&~~\\
\cline{2-3}
&\multicolumn{1}{|c|}{~~}\\
\cline{2-2}
&\multicolumn{1}{|c|}{~~}\\
\cline{2-2}
\end{tabular}
$$

We define a partial ordering on the shapes for diagrams by putting:
$$
\mu=(b,b')\leq \lambda=(a,a')
$$
if and only if:
$$
b_i\leq a_i, \quad\text{for any } i,\qquad b'_j\leq a'_j, \quad\text{for any } j.
$$
Let us remark that $\lambda-\mu$ is not always a shape of a Ferrer diagram, indeed, $\lambda-\mu$ is a shape if and only if $a_m-b_m\geq\sup\{j,~a'_j>b'_j\}$.

This relation defines an ordering on the set of $(m,n)$-shapes. This ordering corresponds to the restriction of the natural ordering on $\Lambda$.

%
%


The semistandard tableaux for $\mathfrak{sl}(m,n)$, are particular filling of Ferrer diagrams.\\


\subsection{Semistandard Young tableaux}


\begin{defn}\label{semistandard}

\

A semistandard tableau for $\mathfrak{sl}(m,n)$ is a filling of a Ferrer diagram with natural numbers $t_{i,j}$ (in the $(i,j)$-box) such that
\begin{itemize}
\item[1.] For every $i$, $j$, $t_{i,j}\leq n+m$,
\item[2.] Along each row, the $t_{i,j}$ are increasing from the left to the right, and strictly increasing if $t_{i,j}>m$:
$$
\text{For all}~~ j,~~t_{i,j}\leq t_{i,j+1},~~\text{and, if}~~t_{i,j}>m,~~\text{then}~~t_{i,j}<t_{i,j+1},
$$
\item[3.] Along each column,  the $t_{i,j}$ are increasing from the top to the bottom, and strictly increasing if $t_{i,j}\leq m$:
$$
\text{For all}~~ i,~~t_{i,j}\leq t_{i+1,j},~~\text{and, if }~~t_{i,j}\leq m,~~\text{then}~~t_{i,j}<t_{i+1,j}.
$$
\end{itemize}

In these relations and everywhere in this section, we use the following convention: if a member in an inequality does not exist, then the inequality holds.\\
\end{defn}

For instance the following tableaux $T$ and $T'$ are semistandard for $\mathfrak{sl}(2,3)$:
$$
T=\begin{array}{rlr}
~~~~~&\begin{array}{|c|c|c|}
\hline
1&1&2\\
\hline
\end{array}\\
&\begin{array}{|c|c|}
2&4\\
\end{array}&~~~~~~~~\\
\hline
&\begin{array}{|c|c|}
3&4\\
\hline
\end{array}&\\
&\begin{array}{|c|}
4\\
\hline
\end{array}&\\
\end{array}
\qquad\qquad T'=
\begin{array}{rlr}
~~~~~~~~&\begin{array}{|c|c|}
\hline
1&2\\
\hline
\end{array}&~~~~~~~~~~~\\
&\begin{array}{|c|}
4\\
\end{array}&\\
\hline
&\begin{array}{|c|}
4\\
\hline
4\\
\hline
\end{array}&\\
\end{array}
$$

We denote $SS^\lambda_{m,n}$ or $SS^\lambda$ the set of semistandard Young tableaux for $\mathfrak{sl}(m,n)$, with shape $\lambda$.\\

If $n=0$, a semistandard tableau for $\mathfrak{sl}(m,0)$ is a usual semistandard Young tableau. If $m=0$, a semistandard tableau for $\mathfrak{sl}(0,n)$ is the transpose of a usual semistandard Young tableau.\\




If $T$ is a semistandard tableau for $\mathfrak{sl}(m,n)$, we define $T^+$ as the tableau formed by the $m$ first rows of $T$ and $T^-$ the tableau whose the $(j-m)^{th}$ row is the row $j$, for $j>m$. We shall write $T=T^+\uplus T^-$:
$$
t^+_{i,j}=t_{i,j}\qquad(i\leq m),\qquad t^-_{i,j}=t_{j+m,i}.
$$

For each semistandard tableau $T$ with shape $\lambda$, we associate, as usual, a vector $e_T$ in $\mathbb S^\lambda$, realized as a submodule in $T(V)$.

First, we choose the standard tableau $St=(s_{ij})$, with shape $\lambda$, defined by:
$$
s_{ij}=i+\sum_{k<j}c_k,
$$ 
if $c_k$ is the height of the column $k$, that is we fill the boxes column by column, from the top to the bottom and from left to right. 

Then, if $(e_1,\dots,e_{n+m})$ is the canonical basis in $V$, and $T=(t_{ij})$ is a semistandard or not semistandard tableau with shape $\lambda$ ($t_{ij}\leq n+m$), we denote $e^T$ the tensor product:
$$
e^T=(e_{t_{11}}\otimes e_{t_{21}}\otimes\dots\otimes e_{t_{c_11}})\otimes(e_{t_{12}}\otimes\dots\otimes e_{t_{c_22}})\otimes\dots\otimes(e_{t_{1A_1}}\otimes\dots\otimes e_{t_{c_{A_1}A_1}}).
$$

Now, if $N=\sum c_k$ is the number of entries in $\lambda$, the group $S_N$ acts on the right on $T^N(V)$ by 
$$
v_1\otimes\dots\otimes v_N\cdot\sigma=\varepsilon_{|v|}(\sigma)v_{\sigma(1)}\otimes\dots\otimes v_{\sigma(n)}.
$$

Denote $a_\lambda$ the sum of the permutations $\sigma$ in $S_N$ which preserve the rows of the standard tableau $St$, similarly denotes $b_\lambda$ the linear combination of the permutations $\tau$ in $S_N$ which preserve the columns of $St$, with the sign of $\tau$ as coefficient. 

The Young symmetrizer associated to our choice of $St$ is $c_\lambda=a_\lambda\cdot b_\lambda$. Finally, we associate to $T$ the vector:
$$
e_T=e^T\cdot c_\lambda.
$$
Remark that with our preceding notations,
$$
e_T\in Sym^{\tilde{a}_1}(\wedge^{m+A'_1}V)\otimes\dots\otimes Sym^{\tilde{a}_{n-1}}(\wedge^{m+A'_{n-1}}V)\otimes Sym^{\tilde{a}_m}(\wedge^mV)\otimes\dots\otimes Sym^{a_1}(V).
$$

For instance, for $\mathfrak{sl}(2,2)$, we consider the semistandard tableau:
$$
T=
\begin{array}{rll}
~~~~~~~~~~~~~~&\begin{array}{|c|c|}
\hline
2&2\\
\hline
3&4
\end{array}&~~~~~~~~~~~~~~~~~\\
\hline
&\begin{array}{|c|}
3\\
\hline
\end{array}&
\end{array}.
$$
Then:
$$
e^T=(e_2\otimes e_3\otimes e_3)\otimes(e_2\otimes e_4),
$$
and 
$$
e^T\cdot a_\lambda=2(e^T-e^{T'}),\quad\text{ where }T'=
\begin{array}{rll}
~~~~~~~~~~~~~~&\begin{array}{|c|c|}
\hline
2&2\\
\hline
4&3
\end{array}&~~~~~~~~~~~~~~~~~\\
\hline
&\begin{array}{|c|}
3\\
\hline
\end{array}&
\end{array}
$$
and finally:
$$
e_T=e^T\cdot c_\lambda=3!2!\left[2(e_2\wedge e_3\wedge e_3)\otimes(e_2\wedge e_4)-2(e_2\wedge e_4\wedge e_3)\otimes(e_2\wedge e_3)\right].
$$

Especially, with our preceding notations and choice of the factor $\alpha_\lambda$ (Proposition \ref{vlambda}), 
$$
v_\lambda=e_{S^0_\lambda},
$$
where $S^0_\lambda$ is the tableau such that $t_{ij}=i$ if $i\leq m$ and $t_{ij}=m+j$ if $i>m$.\\

Berele and Regev prove:

\begin{thm}(\cite{BR})

\

A basis for the space $\mathbb S^\lambda$ is given by the familly of all the vectors $e_T$ for $T$ in $SS^\lambda$.\\
\end{thm}

In the next section, we recall a presentation of $\mathbb S^\bullet$ as the subspace of elements in $T(\bigwedge V)$ satisfying the Garnir relations.\\



\section{Garnir and Pl\"ucker relations}


\

We first saw (see Theorem \ref{shapeslm}) that in the classical $\mathfrak{sl}(m)$ case, the shape algebra is a quotient of $Sym^\bullet(\bigwedge \mathbb C^m)$ by an ideal generated with the Pl\"ucker relations. In the graded case, the generalization of Pl\" ucker relations are the Garnir relations.\\ 

Let us first describe the Garnir relations.\\

Let $C=(v_1,\ldots,v_P)$ and $D=(w_1,\ldots,w_Q)$ be two finite sequence of vectors in $V$. We suppose $P\geq Q$ and put $(v_1,\ldots,v_P,w_1,\ldots,w_Q)=(u_1,\ldots,u_{P+Q})$.\\


\

The Garnir relations take place in $Sym^\bullet(\bigwedge V)$. We consider vectors with the form:
$$
u_C\cdot u_D=(u_1\wedge \ldots \wedge u_P)\cdot(u_{P+1}\wedge\ldots \wedge u_{P+Q}),
$$
where $\cdot$ is for $\otimes$ if $P>Q$ and the symmetric product if $P=Q$. Recall that if $\sigma$ is in $S_{P+Q}$,
$$
(u_C\cdot u_D)\cdot\sigma=\varepsilon_{|u|}(\sigma)(u_{\sigma(1)}\wedge\ldots\wedge
u_{\sigma(P)}).(u_{\sigma(P+1)}\wedge\ldots\wedge u_{\sigma(P+Q)}).
$$
To define the Garnir relations, we consider particular permutations $\sigma=(X'\leftrightarrow Y')$.\\

Let $p\leq P$, $q\leq Q$, such that $p+q>P$. put $X=(v_1,\ldots, v_p)$, $Y=(w_1,\ldots, w_q)$.

A subsequence with $r$ elements $X'\subset X$ is a sequence $(v_{i_1},v_{i_2},\ldots,v_{i_r})$ such that $i_1<i_2<\ldots<i_r$. Denote $s_r(X)$ the set of such subsequences.

If $r\leq inf(p,q)$, let $X'$ be in $s_r(X)$, and $Y'$ in $s_r(Y)$,
$$\aligned
X'&=(v_{i_1},v_{i_2},\ldots,v_{i_r})=(u_{i_1},u_{i_2},\ldots,u_{i_r}),\\
Y'&=(w_{j_1},w_{j_2},\ldots,w_{j_r})=(u_{P+j_1},u_{P+j_2},\ldots,u_{P+j_r}),
\endaligned
$$
we define the permutation $X'\leftrightarrow Y'$ in $S_{P+Q}$ as:
$$\aligned
X'\leftrightarrow Y'&=(i_1,P+j_1)(i_2,P+j_2)\dots(i_r,P+j_r)\\
&=\left(\begin{smallmatrix}
1&\ldots&i_1&\ldots&i_r&\ldots&P&P+1&\ldots&P+j_1&\ldots&P+j_r&\ldots&P+Q\\
1&\ldots&P+j_1&\ldots&P+j_r&\ldots&P&P+1&\ldots&i_1&\ldots&i_r&\ldots&P+Q\\
\end{smallmatrix}\right).\endaligned
$$

By definition, the Garnir relations on the vector $u_C.u_D$ associated to $X$ and $Y$ is:
$$
G_{X,Y}(u_C.u_D)=\sum_{r=0}^{inf(p,q)}~(-1)^r\sum_{\begin{smallmatrix}X'\in s_r(X)\\
Y'\in s_r(Y)\end{smallmatrix}}(u_C\cdot u_D)\cdot(X'\leftrightarrow Y').
$$

\begin{thm}~~ (\cite{KW})

\

As a vector space, the shape algebra $\mathbb S^\bullet$ is the quotient of the symmetric algebra
$$
Sym^\bullet(\wedge V)=\sum_{\tilde{a}_k,\dots,\tilde{a}_1}Sym^{\tilde{a}_k}(\wedge^k V)\otimes\dots\otimes Sym^{\tilde{a}_1}(V),
$$
by the ideal generated by the Garnir relations.\\

More precisely, given $\lambda\in\Lambda_{cov}$, if $\tilde{a}_j$ is the number of columns with height $j$, $\mathbb S^\lambda$ is exactly the image of 
$\bigotimes_j(Sym^{\tilde{a}_j}(\wedge^jV))$ in the quotient.\\
\end{thm}

The graded Pl\"ucker relations are the following: for any $q\geq 1$, we consider $Y=(u_{P+1},\ldots,u_{P+q})$, the relation is:
$$
u_C.u_D=P_q(u_C.u_D)=\sum_{X'\in s_q(C)}(u_C\cdot u_D)\cdot(X'\leftrightarrow Y).
$$
In fact, it is proved in \cite{KW} that the Garnir relations are equivalent to the Pl\"ucker relations, since:

\begin{thm}

\

We keep our notations. Then the following are equivalent:
\begin{itemize}
\item[1)] $u_C.u_D=P_1(u_C.u_D)$ for any $C$, and $D$,
\item[2)] $u_C.u_D=P_q(u_C.u_D)$ for any $C$, any $D$, any $q$,
\item[3)] $G_{C,Y}(u_C.u_D)=0$ for any $C$, any $D$, and any $Y \neq \emptyset$,
\item[4)] $G_{X,Y}(u_C.u_D)=0$ for any $C$, any $D$, any $X$, any $Y$ such that $\#(X\cup Y)>\#C$.\\
\end{itemize}
\end{thm}

If we consider vectors of the form $e^T$ for a given shape $\lambda$, the Pl\"ucker relations imply columns of $T$. More precisely suppose $T$ has $k$ columns with height $c_1,\ldots,c_k$, $e_T=\sum\pm e_{C_1}\cdot\ldots\cdot e_{C_k}$, with $e_{C_j}\in \wedge^{c_j}V$.

For any $j<k$, for any $q>0$, if $Y$ is the $q$-top of the column $C_{j+1}$, with a small abuse of notations, the Pl\"ucker relation is:
$$\aligned
e_T&=\sum\pm e_{C_1}\cdot\ldots\cdot e_{C_{j-1}}\cdot P_q(e_{C_j}\cdot e_{C_{j+1}})\cdot e_{C_{j+2}}\cdot\ldots\cdot e_{C_k}\\
&=\sum_{X\in s_q(C_j)}e_T\cdot (X\leftrightarrow Y)\\
&=\sum_{X\in s_q(C_j)}\varepsilon_{|e_i|}(X\leftrightarrow Y)e_{T\cdot(X\leftrightarrow Y)}\\
&=P_q^{(j)}(e_T).
\endaligned
$$ 

Later one, we shall work with the part $T^-$ of a semistandard tableau $T$, which is under the line. Especially, we shall extract some `trivial' subtableau from $T^-$, exactly as in \cite{ABW}, but modulo a transposition of the tableau. To be complete, we shall prove here that the transpose of the usual Pl\"ucker relations on $(T^-)^t$ hold for $T$, we call them horizontal Pl\"ucker relations, since, on vectors $e^T$, these relations yield from permutations on the rows of the tableau $T$.\\

We look for two successive rows $R_i$, $R_{i+1}$ in a tableau $T$, containing only entries strictly larger than $m$. We denote these entries $s_1,\dots,s_p$ and $t_1,\dots,t_q$, with $p\geq q$, then the corresponding horizontal Pl\"ucker relation is:
$$
e_T=\sum_{j=1}^p e_{T\cdot(s_j\leftrightarrow t_1)}=HP^{(i)}_1(e_T).
$$

\begin{prop}

\

For any tableau $T$, and any rows $R_i,R_{i+1}$ of $T$ whose entries are all strictly larger than $m$, the horizontal Pl\"ucker relation $e_T=HP^{(i)}_1(e_T)$ holds.\\
\end{prop}

\begin{proof}

Suppose first that $T$ contains only the two rows $R_i$, $R_{i+1}$. Then:
$$
e_T=\sum_{(\sigma,\tau)\in S_p\times S_q}\varepsilon(\sigma)\varepsilon(\tau)(e_{s_{\sigma(1)}}\wedge e_{t_{\tau(1)}})\otimes\ldots \otimes (e_{s_{\sigma(q)}}\wedge e_{t_{\tau(q)}})\otimes\ldots \otimes(e_{s_{\sigma(p)}}).
$$
On the other hand, the sum in $HP_1^{(i)}$ is
$$\aligned
HP_1^{(i)}(e_T)&=\sum_{j=1}^p e_{T\cdot(s_j\leftrightarrow t_1)}\\
&=\sum_j\sum_{\begin{smallmatrix}(\sigma,\tau)\in S_p\times S_q\\
\sigma^{-1}(j)=\tau^{-1}(1)\end{smallmatrix}}\varepsilon(\sigma)\varepsilon(\tau)(e_{s_{\sigma(1)}}\wedge e_{t_{\tau(1)}})\otimes\ldots \otimes(e_{t_1}\wedge e_{s_j})\otimes\ldots\\
&+\sum_j\sum_{\begin{smallmatrix}(\sigma,\tau)\in S_p\times S_q\\
\sigma^{-1}(j)\neq\tau^{-1}(1)\end{smallmatrix}}\varepsilon(\sigma)\varepsilon(\tau)(e_{s_{\sigma(1)}}\wedge e_{t_{\tau(1)}})\otimes\ldots 
\otimes(e_{s_{j'}}\wedge e_{s_j})\otimes\ldots 
\endaligned
$$
Since the vectors $e_k$ are odd, in the first sum, $e_{t_1}\wedge e_{s_j}=e_{s_j}\wedge e_{t_1}=e_{\sigma(\sigma^{-1}(j))}\wedge e_{\tau(\tau^{-1}(1))}$, thus the first sum is exactly $e_T$.

In the second sum, we keep the terms for which $j'<j$, getting a quantity $A$ and, for the terms where $j'>j$, we replace $e_{s_{j'}}\wedge e_{s_j}$ by $e_{s_j}\wedge e_{s_{j'}}$ and $\sigma$ by $\sigma'=(j,j')\circ\sigma$, thus we get the terms of $A$, but with an opposite sign, the second sum vanishes.\\

In the general case, if $T^{<i}$ is the top of $T$, containinig the $i-1$ first rows and $T^{>i+1}$ the bottom, containing rows after $i+1$, we can present the $\cdot a_\lambda$ action on $T$ shematically as:
$$
T\cdot a_\lambda=\sum_{(\sigma,\tau)\in S_p\times S_q}\begin{array}{ccc} &T^{<i}\cdot a_\lambda^{<i}&\\
s_{\sigma(1)}&\ldots&s_{\sigma(p)}\\
t_{\tau(1)}&\ldots&t_{\tau(q)}\\
&&\\
&T^{>i+1}\cdot a_\lambda^{>i+1}&
\end{array}
$$
then the above proof works with the same computation, except we have to modify the signs $\varepsilon(\sigma)$ and $\varepsilon(\tau)$ into $\varepsilon'(\sigma)$ and $\varepsilon'(\tau)$, to take into account the vectors corresponding to the rows of $T$ above and under $R_i$ and $R_{i+1}$. But when we use $(j,j')\circ\sigma$, we verify directly that we still have $\varepsilon'((j,j')\circ\sigma)=-\varepsilon'(\sigma)$.\\

The horizontal Pl\"ucker relation holds.\\
\end{proof}


\section{Product on tableaux}


\

The $\mathfrak{sl}(m,n)$ action on the vectors $e_T$ defines an action on the space generated by the semistandard tableaux with a given shape $\lambda$. More precisely, for any $i$, $j$ the matrix $E_{ij}$ with all entries 0, except in the position $(i,j)$, where the entry is 1 acts on $e_T$ ($T=(t_{k,\ell})$ semistandard or not semistandard) as follows:
$$
E_{ij}\cdot e_T=\sum_{t_{k,\ell}=j}e_{T\cdot(t_{k,\ell}\leftrightarrow i)},
$$
then we decompose this sum onto the basis $e_U$, $U$ in $SS^\lambda$, getting:
$$
E_{ij}\cdot e_T=\sum_r x_r e_{U_r}.
$$
We just write this last relation as:
$$
E_{ij}\cdot T=\sum_r x_r U_r.
$$

The algebra structure on $\mathbb S^\bullet$ defines a multiplication (still denoted $\star$) on the vector space generated by semistandard tableaux. If $T$ and $T'$ are two semistandard tableaux, with shape $\lambda$ and $\mu$, we write:
$$
e_T\star e_{T'}=\sum_i x_ie_{U_i},
$$
with $x_i\neq0$ and each $U_i$ is semistandard, with shape $\lambda+\mu$. We thus put
$$
T\star T'=\sum_i x_iU_i.
$$

Let us describe directly the multiplication $\star$ on the space of tableaux.\\

Let us start with two semistandard tableaux $S$ with shape $\lambda$ and $T$ with shape $\mu$. We write $S=S^+\uplus S^-$ and $T=T^+\uplus T^-$. 

Suppose the length of the rows in $S^+$ are $A_1,\dots,A_m$. We define $S^+\circ T^+$ as the tableau $U^+=(u_{ij})$ whose the entries in the row $i$ are $u_{ij}=s_{ij}$, if $j\leq A_i$, and $u_{i(A_i+j)}=t_{ij}$.

Similarly suppose the height of the columns in $S^-$ are $A'_1,\dots, A'_{n-1}$. We define $S^-\circ T^-$ as the tableau $U^-=(u_{(m+i)j})$ whose the entries in the column $j$ are $u_{(m+i)j}=s_{(m+i)j}$, if $i\leq A'_j$, and $u_{(m+A'_j+i)j}=t_{(m+i)j}$.\\

\begin{prop}

\

With the above notation, we put $U=U^+\uplus U^-$, it is a possibly non semistandard tableau with shape $\lambda+\mu$, suppose the decomposition of $e_U$ into the basis $e_{U_i}$, $U_i$ semistandard is 
$$
e_U=\sum_i x_ie_{U_i},
$$
then
$$
S\star T=\sum_i x_iU_i.
$$
\end{prop}

\begin{proof}
By construction, the tableau $U$ is a Young tableau, with a shape $\lambda+\mu$, especially it corresponds to a weight in $\Lambda_{cov}$.

For the moment, note $\star'$ the operation defined in the proposition. Then, by definition of the $\mathfrak{sl}(m,n)$ action, this map is a morphism of modules between $\mathbb S^\lambda\otimes\mathbb S^\mu$ to $\mathbb S^{\lambda+\mu}$:
$$
E_{ij}\cdot U=(E_{ij}\cdot S)\star' T+S\star'(E_{ij}\cdot T).
$$ 

Moreover, if $S=S^0_\lambda$ and $T=S^0_\mu$ are the tableaux associated to the highest weight vector, by construction $U$ is $S^0_{\lambda+\mu}$. Thus it is semistandard and 
$$
e_{S^0_\lambda\star' S^0_\mu}=e_{S^0_{\lambda+\mu}}=e_{S^0_\lambda}\star e_{S^0_\mu}.
$$
By unicity of the multiplication, this proves our proposition: $\star=\star'$.
\end{proof}

Let us remark that, with the preceding notations, for $S=S^0_\lambda$ and any semistandard $T$, with shape $\mu$, the tableau $U$ is semistandard, with shape $\lambda+\mu$, therefore:
$$
S^0_\lambda\star T=U.
$$


\section{Reduced shape algebra}


\

By definition, the reduced shape algebra is the quotient of the shape algebra by the ideal generated by the relations $v_\lambda=1$.\\

\begin{defn}

\

The reduced shape algebra for $\mathfrak{sl}(m,n)$, denoted $\mathbb S_{red}^\bullet=\mathbb S_{red,~(m,n)}^\bullet$ is the quotient of the shape algebra $\mathbb S^\bullet=\mathbb S_{(m,n)}^\bullet$ by the ideal generated by $e_{S^0_\lambda}-1$, for any shape $\lambda$ in $\Lambda_{cov}$:
$$
\mathbb S_{red}^\bullet=\mathbb S^\bullet/\langle v_\lambda-1\rangle.
$$
\end{defn}

Recall $\mathfrak n^+$ is the nilpotent Lie superalgebra of strictly uppertriangular matrices in $\mathfrak{sl}(m,n)$. Since the $\mathfrak n^+$ action is trivial on the $(v_\lambda-1)$, the ideal $\langle v_\lambda-1\rangle$ and the reduced shape algebra are $\mathfrak n^+$ modules. The goal of this paper is a first study of the structure of the $\mathfrak n^+$-module $\mathbb S_{red}^\bullet$.\\

Denote $p$ the canonical projection from $\mathbb S^\bullet$ to $\mathbb S_{red}^\bullet$.\\

\begin{prop}

\

\begin{itemize}
\item[1.] $\mathbb S_{red}^\bullet$ is a locally nilpotent $\mathfrak n^+$-module,
\item[2.] The unique vector $v$ in $\mathbb S_{red}^\bullet$ such that $\mathfrak n^+.v=0$ is up to a factor $p(1)$ and $p(1)\neq0$,
\item[3.] $\mathbb S_{red}^\bullet$ is an indecomposable module,
\item[4.] For each $\lambda$ in $\Lambda_{cov}$, $p$ is an isomorphism of $\mathfrak n^+$-module from $\mathbb S^\lambda$ onto $p(\mathbb S^\lambda)$.\\
\end{itemize}
\end{prop}

\begin{proof}
Recall first that the simple $\mathfrak{sl}(m,n)$ module $\mathbb S^\lambda$ is a weight module ($\mathbb S^\lambda=\sum_{\mu\leq\lambda}(\mathbb S^\lambda)_\mu$), where $(\mathbb S^\lambda)_\mu$ is the weight subspace in $\mathbb S^\lambda$, with weight $\mu$. Now for any non vanishing $v$ in $(\mathbb S^\lambda)_\mu$, there is a finite sequence $(X_1,\ldots, X_k)$ of elements in $\mathfrak n^+$, such that $X_1\ldots X_k.v=v_\lambda$. 

\noindent
1. $\mathbb S_{red}^\bullet$ is generated by the vectors $w=w_\mu^\lambda$ in $p((\mathbb S^\lambda)_\mu)$. For these vectors there is $k$ such that $(\mathfrak n^+)^k.w=p(v_\lambda)$, thus $(\mathfrak n^+)^{k+1}.w=0$, $\mathbb S_{red}^\bullet$ is a locally nilpotent $\mathfrak n^+$-module.\\

\noindent
2. Suppose $p(1)=0$, that means there are some non vanishing weight vectors $w_{\mu,\nu}^\lambda$ in some $(\mathbb S^\mu)_{\mu-\nu}$, such that:
$$
1=\sum_{\lambda\neq0,\mu,\nu}w_{\mu,\nu}^\lambda\star v_\lambda-w_{\mu,\nu}^\lambda\in \sum_{\lambda,\mu,\nu}\mathbb S^{\lambda+\mu}+\mathbb S^{\mu}.
$$
But 1 is a dominant weight vector in the shape algebra, the dominant weight vectors in the right hand side are $w_{\mu,0}^\lambda=C_\mu^\lambda v_\mu$ and $w_{\mu,0}\star v_\lambda$, thus we can write:
$$
1-\sum_{\lambda,\mu}C^\lambda_\mu v_\mu\star(v_\lambda-1)=\sum_{\lambda,\mu,\nu\neq 0}w^\lambda_{\mu,\nu}\star(v_\lambda-1)=u.
$$
For any $X$ in $\mathfrak n^+$, $X.u=0$, thus $u$ is a linear combination of $v_\rho$, since no weight vector in the right hand side is dominant, the only possibility is $u=0$. We have
$$
1=\sum_{\lambda>0,\mu\geq0}C^\lambda_\mu v_\mu\star(v_\lambda-1)=\sum_{\lambda,\mu}C^\lambda_\mu (v_{\lambda+\mu}-v_\mu)=\sum_{\lambda\geq\rho>0}C^\lambda_{\rho-\lambda}v_\rho-\sum_{\lambda>0,\mu\geq0}C^\lambda_\mu v_\mu.
$$
We get the relations:
$$
-1=\sum_{\lambda>0}C_0^\lambda,\quad \sum_{\lambda}C_{\mu-\lambda}^\lambda=\sum_\lambda C_\mu^\lambda\quad(\mu>0).
$$
We finally sum up all these relations, getting:
$$
-1+\sum_{\lambda,\mu\geq0}C^\lambda_\mu=\sum_{\lambda,\mu\geq0}C^\lambda_\mu,
$$
which is impossible.\\

Let now $w=p(v)$ be a vector in the reduced shape algebra which satisfies $\mathfrak n^+.w=0$, then with a preceding argument,
$$
v=\sum_\mu A^\mu v_\mu+\sum_{\lambda,\nu}C^\lambda_\nu v_\nu\star(v_\lambda-1).
$$
Therefore, $w=p(v)=(\sum_\mu A^\mu)p(1)$ is a multiple of the vector $p(1)$.\\

\noindent
3. Let $W$ be a $\mathfrak n^+$-submodule in the reduced shape algebra $\mathbb S^\bullet_{red}$, then since this module is locally nilpotent, it contains a non trivial vector annhilated by $\mathfrak n^+$, thus $p(1)$ belongs to $W$. $\mathbb S^\bullet_{red}$ is an indecomposable $\mathfrak n^+$-module.\\

\noindent
4. We consider the kernel $K$ of the restriction of $p$ to $\mathbb S^\lambda$. It is a nilpotent $\mathfrak n^+$-module, thus, if it is not trivial, it contains a non vanishing vector annhilated by $\mathfrak n^+$. Since there is only one such vector in $\mathbb S^\lambda$, $K$ contains $v_\lambda$, thus $p(v_\lambda)=p(1)=0$, which is impossible, $p$ is an isomorphism of $\mathfrak n^+$-modules from $\mathbb S^\lambda$ onto $p(\mathbb S^\lambda)$.\\
\end{proof}

\begin{cor}\label{stratif}

\

If $\mu\leq\lambda$, then $p(\mathbb S^\mu)$ is a submodule of $p(\mathbb S^\lambda)$.\\
\end{cor}

\begin{proof}
Using the notations in the end of section \ref{shape}, we have:
$$
p(\mathbb S^\mu)=p(v_{\lambda-\mu+k\eta}\star\mathbb S^\mu)\subset p(v_{k\eta}\star\mathbb S^\lambda)=p(\mathbb S^\lambda).
$$
\end{proof}

We shall now build a combinatorial basis, for the reduced shape algebra $\mathbb S_{red}^\bullet$ which respects the above stratification of indecomposable modules, using particular semistandard tableaux, called quasistandard Young tableaux.\\


\section{Quasistandard tableaux}


\

Let us first define what is a trivial tableau.\\

\begin{defn}

\

We say that a tableau $S^+$ is trivial if it contains only columns with heights at most $m$, and its entries are $s^+_{ij}=i$ for any $i$ and $j$.

We say that a tableau $S^-$ is trivial if it contains no rows with index $1,\dots,m$, and if its entries are $s^-_{m+i,j}=m+j$, for any $i$, and $j$. (Remark that $S^-$ is not a Young tableau).

We say that a Young tableau $S$ is trivial if it is $S=S^+\uplus S^-$, with $S^+$ and $S^-$ trivial.\\
\end{defn}

\begin{defn}\label{pairextractable}

\

Suppose $T=T^+\uplus T^-=(t_{i,j})$ is a semistandard tableau. Let $S^+$, $S^-$ be two trivial tableaux, let $B_i$ be the length of the row $i$ in $S^+$ and $B'_j$ the height of the column $j$ in $S^-$. We say that the pair of trivial subtableaux $(S^+,S^-)$ is extractable from $T$ if:
\begin{itemize}
\item[$i.$] $S^\pm$ is on the top and left subtableau in $T^{\pm}$, 

\item[$ii.$] the tableau $U$ defined by:
$$
u_{ij}=t_{i,j+B_i}, \quad(i\leq m),\text{ and }u_{ij}=t_{i+B'_j,j}, \quad(i>m)
$$
is a semistandard tableau.\\
\end{itemize}
\end{defn}

Let $T$ be a semistandard tableau, with shape $(a,a')$. Denote $A_i$ the length of the row $i$ ($i\leq m$) and $m+A'_j$ the height of the column $j$ ($j<n$ and the column has height larger than $m$). We consider a pair $(S^+,S^-)$ of trivial tableaux, with respective shape $b=(b,0)$ and $b'=(0,b')$. We denote $B_i$, (resp. $B'_j$) the length (resp. the height) of the rows (resp. the columns) of $S^+$ (resp. $S^-$). $(S^+,S^-)$ is extractable from $T$ semistandard means the following conditions hold:

\begin{itemize}
\item[E1.] for $i\leq m$, $b_i\leq a_i$, and $t_{i,j}=j$ if $j\leq B_i$, and
$$
t_{i,j+B_i}\leq t_{i+1,j+B_{i+1}}, ~~\text{ and }~~t_{i,j+B_i}<t_{i+1,j+B_{i+1}}\text{ if }t_{i,j+B_i}\leq m,
$$

\item[E2.] for $j<n$, $b'_j\leq a'_j$, and for $i>m$, $t_{i,j}=m+j$ if $i\leq B_j$, and
$$
t_{i+B'_j,j}< t_{i+B'_{j+1},j+1},
$$

\item[E3.] $\widetilde{(a-b)}_m=(a_m-b_m)-\sup\{j,~a'_j>b'_j\}\geq0$,

\item[E4.] for all $j$, $t_{m,j+B_m}\leq t_{m+1+B'_j,j}$.\\
\end{itemize}

\begin{exple}

For instance, for $\mathfrak{sl}(2,3)$, the tableau:
$$
T=\begin{tabular}{cccccc}
\cline{2-5}
\hskip 5cm&\multicolumn{1}{|c|}{$1$}&\multicolumn{1}{|c|}{$1$}&\multicolumn{1}{|c|}{$2$}&\multicolumn{1}{|c|}{$2$}&\hskip 5cm\\
\cline{2-5}
\multicolumn{1}{c}{~~~~~~}&\multicolumn{1}{|c|}{$2$}&\multicolumn{1}{|c|}{$3$}&\multicolumn{1}{|c|}{$4$}&~~&\multicolumn{1}{c}{~~~~~~~~$_2$}\\
\hline
&\multicolumn{1}{|c|}{$3$}&\multicolumn{1}{|c|}{$4$}\\
\cline{2-3}
&\multicolumn{1}{|c|}{$3$}&\multicolumn{1}{|c|}{$5$}\\
\cline{2-3}
&\multicolumn{1}{|c|}{$4$}&\multicolumn{1}{|c|}{$5$}\\
\cline{2-3}
&\multicolumn{1}{|c|}{$4$}&\multicolumn{1}{c}{~~}\\
\cline{2-2}
&\multicolumn{1}{|c|}{$5$}&\multicolumn{1}{c}{~~}\\
\cline{2-2}
\end{tabular},
$$
contains the trivial extractable pair:
$$
(S^+,S^-)=\begin{tabular}{cccccc}
\cline{2-3}
\hskip 5cm&\multicolumn{1}{|c|}{$1$}&\multicolumn{1}{|c|}{$1$}&\multicolumn{1}{c}{~~}&\multicolumn{1}{c}{~~}&\hskip 5cm\\
\cline{2-3}
\multicolumn{1}{c}{~~~~~~}&\multicolumn{1}{|c|}{$2$}&\multicolumn{1}{c}{~~}&\multicolumn{1}{c}{~~}&~~&\multicolumn{1}{c}{~~~~~~~~$_2$}\\
\hline
&\multicolumn{1}{|c|}{$3$}&\multicolumn{1}{|c|}{$4$}\\
\cline{2-3}
&\multicolumn{1}{|c|}{$3$}&\multicolumn{1}{c}{~~}\\
\cline{2-2}
\end{tabular}.
$$

\end{exple}

\begin{lem}

\

Among the pairs of trivial extractable tableaux in $T$, there is an unique largest one.\\
\end{lem}

\begin{proof}

Suppose that $(S^{1+},S^{1-})$ with respective shapes $b^1,(b^1)'$ and $(S^{2+},S^{2-})$ with respective shapes $b^2,(b^2)'$ are trivial pairs  extractable from $T$, then the pair $S^+=S^{1+}\cup S^{2+}$, with shape $b$, $S^-=S^{1-}\cup S^{2-}$, with shape $b'$, is trivial, let us prove it is extractable:

\begin{itemize}
\item[1.] For $i\leq m$, $B_i=\sup\{B^1_i,B^2_i\}$, thus $b_i=B_i-B_{i+1}\leq\sup\{b^1_i,b^2_i\}$ and the relation E1. holds in any case. For instance, if $B_i=B^1_i$ and $B_{i+1}=B^2_{i+1}$, then 
$$
t_{i,j+B_i}=t_{i,j+B^1_i}\leq t_{i+1,j+B^1_{i+1}}\leq t_{i+1,j+B^2_{i+1}}=t_{i+1,j+B_{i+1}},
$$
and, if $t_{i,j+B_i}\leq m$, 
$$
t_{i,j+B_i}=t_{i,j+B^1_i}< t_{i+1,j+B^1_{i+1}}\leq t_{i+1,j+B^2_{i+1}}=t_{i+1,j+B_{i+1}}.
$$

\item[2.] E2 holds with the same argument: $b'_j\leq a'_j$, for any $j$, and if for instance $B_j=B^1_j$ and $B_{j+1}=B^2_{j+1}$, then 
$$
t_{i+B'_j,j}=t_{i+(B^1)'_j,j}< t_{i+(B^1)'_{j+1},j+1}\leq t_{i+(B^2)'_{j+1},j+1}=t_{i+B'_{j+1},j+1},
$$

\item[3.] Remark that $\{j,~A'_j>B'_j\}=\{j,~A'_j>(B^1)'_j\}\cap\{j,~A'_j>(B^2)'_j\}$ and $B_m=b_m$. Suppose for instance that $B_m=B^1_m$, then
$$
\widetilde{(a-b)}_m=a_m-b_m-\sup\{j,~A'_j>B'_j\}\geq a_m-b^1_m-\sup\{j,~A'_j>(B^1)'_j\}\geq0.
$$
Relation E3 holds.

\item[4.] Suppose $B_m=B^1_m$. For any $j$, if $B'_j=(B^1)'_j$, the inequality E4. holds, if $B'_j=(B^2)'_j$,
$$
t_{m,j+B_m}=t_{m,j+B^1_m}\leq t_{m+1+(B^1)'_j,j}\leq t_{m+1+(B^2)'_j,j}=t_{m+1+B'_j,j},
$$
E4 is still holding.
\end{itemize}

Therefore there is a largest trivial extractable pair $(S^+,S^-)$ inside $T$.\\
\end{proof}

Let us denote $(S^+_T,S^-_T)$ the largest trivial extractable pair in $T$.\\

\begin{rema}\label{maxmoins}

\

Let $(S^+,S^-)$ be a trivial extractable pair in the semistandard tableau $T$ for $\mathfrak{sl}(m,n)$. Then $S'=(S^-)^t$ is extractable (for $\mathfrak{sl}(n)$) from the tableau $T'=(T^-)^t$ semistandard for $\mathfrak{sl}(n)$. It is easy to prove that if a tableau $S_1'$ is trivial extractable for $T'$, and contains $S'$, then the pair $(S^+,(S'_1)^t)$ is trivial rextractable from $T$.

Therefore, if $S_{T'}$ is the maximal trivial extractable tableau for $T'$, we have
$$
(S_{T'})^t=S^-_T.
$$
\end{rema}

\begin{defn}

\

A semistandard tableau $T$ is quasistandard if the unique trivial extractable pair in $T$ is the empty pair $(\emptyset,\emptyset)$.

Denote $QS^\lambda_{(m,n)}=QS^\lambda$ the set of all quasi standard tableaux with shape $\lambda$.

For any semistandard tableau, we define the tableau $push(T)$ by suppressing from $T$ the largest trivial extractable pair, that is if $b=(b_i)$, $b'=(b'_j)$ are the shapes of $S_T^+$ and $S_T^-$, $push(T)=(u_{i,j})$ is the semistandard tableau defined by
$$
u_{i,j}=t_{i,j+B_i}\qquad(i\leq m),\quad\text{ and }\quad u_{i,j}=t_{i+B'_j,j}\quad (i>m).
$$
\end{defn}

\

Remark that if $T$ belongs to $SS^{(a,a')}$ and $b$ (resp. $b'$) is the shape of $S_T^+$ (resp. $S_T^-$), then $push(T)$ is in $SS^{(a-b,a'-b')}$ and $\mu=(a-b,a'-b')\leq(a,a')=\lambda$.\\

Let us now prove that $push$ is a bijective mapping between $SS^\lambda$ and $\bigsqcup_{\mu\leq\lambda}QS^\mu$.\\

\begin{prop}

\

Consider the restriction $push^\lambda$ of $push$ to $SS^\lambda$. Then, $push^\lambda$ is a bijection from $SS^\lambda$ onto $\bigsqcup_{\mu\leq \lambda}QS^\mu$.\\
\end{prop}

\begin{proof}

Let $T=(t_{i,j})$ be a semistandard tableau of shape $\lambda=(a,a')$, let $(S^+_T,S^-_T)$ the largest trivial extractable pair in $T$. Suppose the shape of $S^+_T$ (resp. $S^-_T$) is $b$ (resp. $b'$).\\

We first prove that $push(T)$ is in $QS^{(a-b,a'-b')}$. Suppose $push(T)=(u_{i,j})$ contains a non empty trivial extractable pair $(S^+,S^-)$, with respective shapes $c,~c'$. Then:
$$\aligned
&u_{i,j}=i=t_{i,j+B_i}\quad(i\leq m,~~j\leq C_i)\qquad&\text{ or }\quad &t_{i,j}=i\quad(i\leq m,~~j\leq B_i+C_i)\\
&u_{i,j}=j+m=t_{i+B'_j,j}\quad(m<i\leq C'_j)\qquad&\text{ or }\quad &t_{i,j}=j+m\quad(m<i\leq B'_j+C'_j).
\endaligned
$$
That means the trivial pair $(\tilde{S}^+,\tilde{S}^-)$ with respective shapes $d,~d'$ such that $D_i=B_i+C_i$, $D'_j=B'_j+C'_j$ is a subtableau of $T$. Moreover, for all $i\leq m$ and $j$,
$$
t_{i,j+D_i}=u_{i,j+C_i}\leq u_{i+1,j+C_{i+1}}=t_{i+1,j+D_{i+1}},\quad\text{ and }\quad t_{i,j+D_i}<t_{i+1,j+D_{i+1}}~~\text{ if }~~t_{i,j+D_i}\leq m,
$$
and, for any $i>m$,
$$
\quad t_{i+D'_j,j}=u_{i+C'_j,j}< u_{i+C'_{j+1},j+1}=t_{i+D'_{j+1},j+1}.
$$
Similarly:
$$
A_m-D_m-\sup\{j,~A'_j>D'_j\}=(A_m-B_m)-C_m-\sup\{j,~A'_j-B'_j>C'_j\}\geq 0,
$$
and for all $j$, 
$$
t_{m,j+D_m}=u_{m,j+C_m}\leq u_{m+1+C'_j,j}=t_{m+1+D'_j,j}.
$$
All these relations prove that $(\tilde{S}^+,\tilde{S}^-)$ is a trivial extractable pair in $T$, which is strictly larger than $(S_T^+,S_T^-)$, which is impossible. $push(T)$ is quasistandard.\\

Let us define now, for any $\mu=(b,b')\leq\lambda=(a,a')$ in $\Lambda_{cov}$, the `$pull^{\lambda-\mu}$' map as the following: if $U=U^+\uplus U^-$ is a semistandard tableau with shape $\mu$, we denote $(S^+,S^-)$ the trivial pair with shapes $a-b,a'-b'$ and define:
$$
pull^{\lambda-\mu}(U)=(S^+\star U^+)\uplus(S^-\star U^-).
$$
That is we define $T=pull^{\lambda-\mu}(U)$ as the tableau $t_{ij}$ with, if $i\leq m$,
$$
t_{i,j}=i\quad(j\leq A_i-B_i),\qquad t_{i,j+A_i-B_i}=u_{i,j}\quad(j\leq B_i),
$$
and if $i>m$,
$$
t_{i,j}=m+j\quad(i\leq A'_j-B'_j),\qquad t_{i+A'_j-B'_j,j}=u_{i,j}\quad(i\leq B'_j).
$$
It is clear that $T$ is a Young tableau with shape $\lambda$. It is easy to verify that $T$ belongs to $SS^\lambda$, and of course, $(S^+,S^-)$ is a trivial extractable pair in $T$.

Now, if $(S^+,S^-)\subsetneq (S^+_T,S^-_T)$, then by construction, $U$ contains a non empty trivial extractble pair. Therefore, if $U$ is quasistandard, then $(S^+,S^-)$ is the largest trivial extractable pair in $T$, $U=push(pull^{\lambda-\mu}(U))$, and $push$ is onto.\\

Finally, for any semistandard tableau $T$ with shape $\lambda$, if the shape of $push(T)$ is $\mu$, then, by definition, $pull^{\lambda-\mu}(push(T))=T$ and $push$ is one-to-one.\\
\end{proof}

We now get the wanted basis for $\mathbb S^\bullet_{red}$, {\i.e.} the diamond cone for $\mathfrak{sl}(m,n)$.\\

\begin{rema}

\

To prove the quasistandard tableaux give a generating system for the reduced shape algebra, we could use the argument in \cite{ABW}:

A consequence of the horizontal Pl\"ucker relations is that any semistandard tableau $T'$ for $\mathfrak{sl}(n-1)$ can be written $S_{T'}\star U'+\sum_{T'_j<T'} T'_j$ where $<$ is a total ordering of tableaux, coming back to $T$ this gives, thank to remark \ref{maxmoins}:
$$
T=(\emptyset\uplus S^-_T)\star T^+\uplus U^-+\sum_{T_j<T}T_j.
$$

As a consequence of the Pl\"ucker relations $P_q$ with $q\leq m$, $T^+$ can be similarly written as $S_T^+\star U^++\sum_{T^+_k<T^+}T^+_k$, this gives finally:
$$
T=(S^+_T\uplus S^-_T)\star U+\sum_{R_i<T}R_i.
$$
With this result, it is easy to prove that $\{p(e_U),~U\in QS^\bullet\}$ generates the reduced shape algebra.\\
\end{rema}

We prefer to now present a direct and simple proof.\\

\begin{thm}

\

The set of all quasistandard tableaux for $\mathfrak{sl}(m,n)$ labels a basis for the reduced shape algebra $\mathbb S_{red}^\bullet$.

More precisely, for any $\lambda$ in $\Lambda_{cov}$, a basis for $p(\mathbb S^\lambda)$ is given by the family 
$$
\Big\{p(e_U),~~U\in\bigcup_{\mu\leq\lambda}QS^\mu\Big\}.
$$
\end{thm}

\begin{proof}

First, for any $T$ semistandard, with shape $\lambda$, consider the largest trivial extractable pair $(S^+_T,S^-_T)$ for $T$. Recall that $S^+_T\uplus S^-_T$ can be not a Young tableau, however, if $\mu$ is the shape of $push(T)$ and $\eta$ is the shape $((0,\ldots,1),(0,\ldots,0))$, there is a natural number $k$ such that $(S^0_{k\eta}\star S^+_T)\uplus S^-_T$ is the (trivial) Young tableau $S^0_{k\eta+\lambda-\mu}$, with shape $k\eta+\lambda-\mu$. It is easy to show that the largest trivial extractable pair in $S^0_{k\eta}\star T$ is $((S^0_{k\eta}\star S^+_T), S^-_T)$, thus
$$
S^0_{k\eta}\star T=S^0_{k\eta+\lambda-\mu}\star push(T),
$$
this proves that $p(e_T)=p(e_{push(T)})$ in the quotient, and a generating system for $p(\mathbb S^\lambda)$ is labelled by $\bigcup_{\mu\leq\lambda} QS^\mu$.\\

Fix a shape $\lambda$ and suppose that, in the reduced shape algebra $\sum a_j p(e_{U_j})=0$, where $a_j\in\mathbb C$, and the $U_j$ are distinct quasistandard tableaux with shape $\mu_j\leq\lambda$. Then the $T_j=pull^{\lambda-\mu_j}(U_j)$ are distinct semistandard tableaux, with shape $\lambda$, and 
$$
p(\sum_ja_j e_{T_j})=\sum_ja_jp(e_{push(T_j)})=\sum_ja_jp(e_{U_j})=0.
$$
But we saw that $p$ is one-to-one on $\mathbb S^\lambda$, this implies:
$$
\sum_j a_j e_{T_j}=0.
$$
Since the $T_j$ are distinct tableaux, $a_j=0$ for any $j$ and the $p(e_{U_j})$ are linearly independent in the reduced shape algebra.\\
\end{proof}

Clearly, the basis given in the theorem is well adapted to the stratification of the $\mathfrak n^+$-modules $p(\mathbb S^\lambda)$ in $\mathbb S_{red}^\bullet$ (see Corollary \ref{stratif}).\\

However, fix a shape $\mu$ and let $\lambda\geq \mu$ a shape such that $\lambda-\mu$ is a shape. Then the $pull^{\lambda-\mu}$ map is simply $U\mapsto S^0_{\lambda-\mu}\star U$, this map corresponds to the injective morphism of $\mathfrak n^+$-modules, from $\mathbb S^\mu$ into $\mathbb S^\lambda$, given by $v\mapsto v_\lambda\star v$. Its inverse mapping can be denoted
$$
push^{\leftarrow \lambda}:pull^{\lambda-\mu}(QS^\mu)\longrightarrow QS\mu.
$$

It is possible to define directly the map $push^{\leftarrow\lambda}$ as the extraction from the semistandard tableau $T$ of a trivial Young tableau:

We say that a trivial Young tableau $S^0$ is extractable from $T$, if it is a trivial tableau $S^0=S^+\uplus S^-$, $S^+$ and $S^-$ trivial, and if $T=S^0\star U$, with $U$ a Young tableau. Among the trivial extractable tableaux, there is a largest one $S^0_{max}$, then $push^{\leftarrow\lambda}(T)$ is defined by
$$
T=S^0_{max}\star push^{\leftarrow\lambda}(T).
$$

Now, if $\lambda-\mu$ is not a shape, $push^{\leftarrow\lambda}(SS^\lambda)\cap SS^\mu=\emptyset$, and there is no natural morphism of $\mathfrak n^+$-module from $\mathbb S^\mu$ into $\mathbb S^\lambda$, but only the morphism from $\mathbb S^{\mu}$ into $\mathbb S^{k\eta+\lambda}$ given by $v\mapsto v_{k\eta+\lambda-\mu}\star v$ and described as at the end of the Section \ref{shape}.\\

\begin{exple}

\

Consider the Lie super algebra $\mathfrak{sl}(1,2)$, and the shape $\lambda=((2),(1))$. Then we get the following picture for the base of $p(\mathbb S^\lambda)$

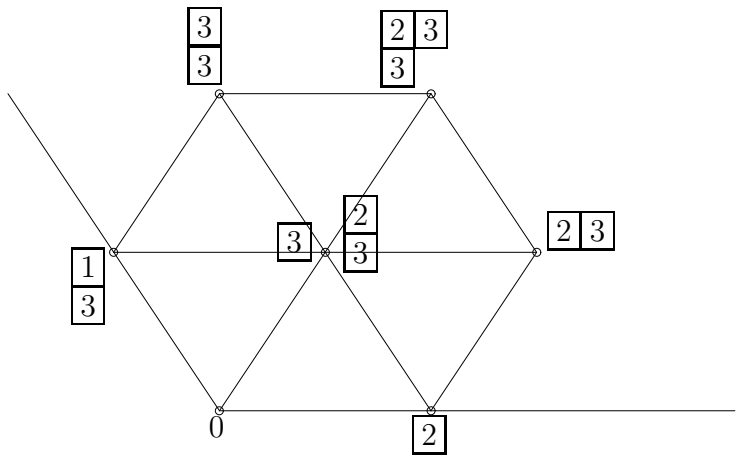
\begin{center}
\begin{picture}(180,120)(200,70)

\path(255,30)(335,30) \path(255,150)(215,90)
\path(215,90)(375,90) \path(450,30)(335,30)
\path(175,150)(255,30) \path(335,150)(255,150) 
\path(335,150)(255,30) \path(255,150)(335,30)
\path(335,150)(375,90) \path(375,90)(335,30)

 \put(335,150){\circle{3}}
\put(255,30){\circle{3}}
\put(215,90){\circle{3}} \put(335,30){\circle{3}}
\put(295,90){\circle{3}} \put(255,150){\circle{3}}
\put(375,90){\circle{3}}

\put(194,74) {\renewcommand{\arraystretch}{0.7}$\begin{array}{l}
\framebox{$1$}\\
\framebox{$3$}\\
\end{array}$}

\put(297,94){\renewcommand{\arraystretch}{0.7}$\begin{array}{l}
\framebox{$2$}\\
\framebox{$3$}\\
\end{array}$}

\put(238,165){\renewcommand{\arraystretch}{0.7}$\begin{array}{l}
\framebox{$3$}\\
\framebox{$3$}\\
\end{array}$}

\put(311,164){\renewcommand{\arraystretch}{0.7}$\begin{array}{l}
\framebox{$2$}\framebox{$3$}\\
\framebox{$3$}\\
\end{array}$}

\put(328,17){\framebox{$2$}}

\put(277,90){\framebox{$3$}}


\put(374,95){\renewcommand{\arraystretch}{0.7}$\begin{array}{l}
\framebox{$2$}\framebox{$3$}\\
\end{array}$}

\put(251,20){$0$}
\end{picture}
\end{center}

\vskip 3cm

Similarly, for $\mathfrak{sl}(2,1)$ and $\lambda=((1,1),)$, we get the picture:

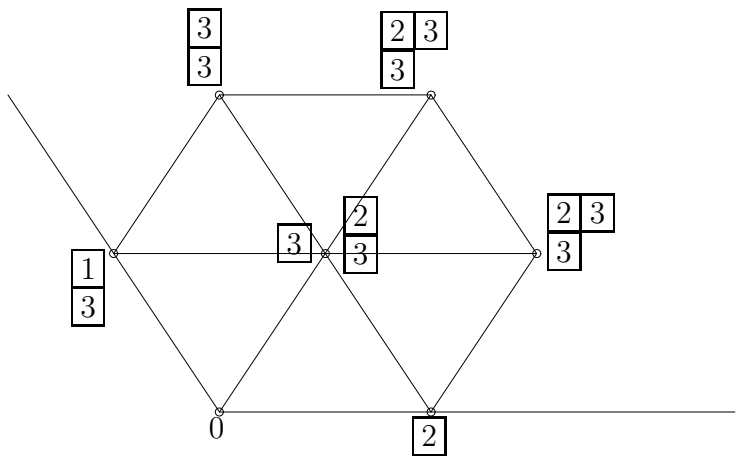
\begin{center}
\begin{picture}(180,120)(200,70)

\path(255,30)(335,30) \path(255,150)(215,90)
\path(215,90)(375,90) \path(450,30)(335,30)
\path(175,150)(255,30) \path(335,150)(255,150) 
\path(335,150)(255,30) \path(255,150)(335,30)
\path(335,150)(375,90) \path(375,90)(335,30)

 \put(335,150){\circle{3}}
\put(255,30){\circle{3}}
\put(215,90){\circle{3}} \put(335,30){\circle{3}}
\put(295,90){\circle{3}} \put(255,150){\circle{3}}
\put(375,90){\circle{3}}

\put(194,74) {\renewcommand{\arraystretch}{0.7}$\begin{array}{l}
\framebox{$1$}\\
\framebox{$3$}\\
\end{array}$}

\put(297,94){\renewcommand{\arraystretch}{0.7}$\begin{array}{l}
\framebox{$2$}\\
\framebox{$3$}\\
\end{array}$}

\put(238,165){\renewcommand{\arraystretch}{0.7}$\begin{array}{l}
\framebox{$3$}\\
\framebox{$3$}\\
\end{array}$}

\put(311,164){\renewcommand{\arraystretch}{0.7}$\begin{array}{l}
\framebox{$2$}\framebox{$3$}\\
\framebox{$3$}\\
\end{array}$}

\put(328,17){\framebox{$2$}}

\put(277,90){\framebox{$3$}}


\put(374,95){\renewcommand{\arraystretch}{0.7}$\begin{array}{l}
\framebox{$2$}\framebox{$3$}\\
\framebox{$3$}
\end{array}$}

\put(251,20){$0$}
\end{picture}
\end{center}
\end{exple}

\vskip 2cm

\section{Super jeu de Taquin}


\

In this section, we define the `super jeu de taquin', {\sl i.e.} the notion of jeu de taquin for $\mathfrak{sl}(m,n)$-tableaux.\\

Let $\lambda=(a,a')$ be a $(m,n)$-shape and $\mu=(b,b')\leq\lambda$, such that $\lambda-\mu$ is a shape. We say that a Ferrer diagram $F=F^+\uplus F^-$ is a skew diagram with shape $\lambda-\mu$ if the shape of $F$ is obtained by suppressing the empty boxes in $\mu$ from $\lambda$. More precisely, denoting $F^\mu=(F^\mu)^+\uplus(F^\mu)^-$ (resp. $F^\lambda=(F^\lambda)^+\uplus(F^\lambda)^-$) a Ferrer diagram with shape $\mu$ (resp. $\lambda$), we see $(F^\mu)^+$, resp. $(F^\mu)^-$ as the corresponding subset at the top and left in $(F^\lambda)^+$, resp. $(F^\lambda)^-$, then $F^\pm=(F^\lambda)^\pm\setminus (F^\mu)^{\pm}$.

For instance, for $\mathfrak{sl}(2,3)$, if $\lambda=((1,3),(2,3))$ and $\mu=((1,1),(1,1))$, the skew Ferrer diagram $F$ with shape $\lambda-\mu$ is:
$$
F=\begin{tabular}{cccccc}
\cline{4-5}
\hskip 3cm&~~&~~&\multicolumn{1}{|c|}{~~}&\multicolumn{1}{|c|}{~~}&\hskip 3cm\\
\cline{3-5}
&&\multicolumn{1}{|c|}{~~}&\multicolumn{1}{|c|}{~~}&\multicolumn{1}{c}{~~}&\multicolumn{1}{c}{~~~~~~$_m$}\\
\hline
&\multicolumn{1}{c}{~~}&\multicolumn{1}{c}{~~}&&&\\
\cline{3-3}
&\multicolumn{1}{c}{~~}&\multicolumn{1}{|c|}{~~}&&&\\
\cline{2-3}
&\multicolumn{1}{|c|}{~~}&\multicolumn{1}{|c|}{~~}&&&\\
\cline{2-3}
&\multicolumn{1}{|c|}{~~}&&&&\\
\cline{2-2}
&\multicolumn{1}{|c|}{~~}&&&&\\
\cline{2-2}
\end{tabular}
$$

For such a skew Ferrer diagram $F$, we define the outer corners as the box $(i,j)$ such that $(i,j)$ is a box in $(F^\mu)^\pm$, but $(i+1,j)$ and $(i,j+1)$ are not in $(F^\mu)^\pm$. Similarly we define inner corner for $F^\pm$ as the boxes $(i,j)$ which are in $F^\pm$ and such that $(i+1,j)$ and $(i,j+1)$ are not in $F^\pm$.

For future purpose, we put a total ordering on the set $Out(F)=Out(F^+)\cup Out(F^-)$ of outer corners of $F$ by putting:
$$\aligned
Out(F^+)&=\{c_1=(i_1,j_1)<\dots<c_p=(i_p,jp)\},\\
Out(F^-)&=\{c'_1=(i'_1,j'_1)<\dots<c'_q=(i'_q,j'_q)\},
\endaligned
$$
if $j_1<j_2<\dots<j_p$ and $i'_1<i'_2<\dots<i'_q$, moreover for any $r$ and $s$, $c_r<c'_s$.

For instance, in the preceding example, $Out(F)=\{(2,1)<(1,2)<(3,2)<(4,1)\}$.

We call semistandard skew tableau $T$ the filling of a skew Ferrer diagram by entries $t_{i,j}$ in $\{1,\dots,m+n\}$ such that for all $j$, $t_{i,j}\leq t_{i,j+1}$, and, if $t_{i,j}>m$, then $t_{i,j}<t_{i,j+1}$, for all $i$, $t_{i,j}\leq t_{i+1,j}$, and, if $t_{i,j}\leq m$, then $t_{i,j}<t_{i+1,j}$.

To define the super jeu de taquin on $T$, we put a star in a outer corner $c$ of the diagram of $T$, and push this star, step by step from the outer corner to a sequence of boxes through the following rules:

Suppose the star is at the box $(i,j)$.
\begin{itemize}
\item[1.] If $i>m$, we try to push to the bottom:
\subitem a. if the box $(i+1,j)$ exists and the box $(i,j+1)$ does not exist or $t_{i+1,j}<t_{i,j+1}$, we put the star in the box $(i+1,j)$ and the entry $t_{i+1,j}$ in the box $(i,j)$. The other entries are not modified.
\subitem b. if the box $(i,j+1)$ exists and the box $(i+1,j)$ does not exist or $t_{i+1,j}\not< t_{i,j+1}$, we put the star in the box $(i,j+1)$ and the entry $t_{i,j+1}$ in the box $(i,j)$. The other entries are not modified.
\subitem c. if the two boxes $(i+1,j)$ and $(i,j+1)$ do not exist (we say that $(i,j)$ is an inner corner for $T$), we suppress the box $(i,j)$ and the star.

\item[2.] If $i\leq m$, we try to push to the right:
\subitem a. if the box $(i,j+1)$ exists and the box $(i+1,j)$ does not exist or $t_{i,j+1}<t_{i+1,j}$ or $t_{i,j+1}=t_{i+1,j}>m$, we put the star in the box $(i,j+1)$ and the entry $t_{i,j+1}$ in the box $(i,j)$. The other entries are not modified.
\subitem b. if the box $(i+1,j)$ exists and the box $(i,j+1)$ does not exist or $t_{i,j+1}>t_{i+1,j}$ or $t_{i,j+1}=t_{i+1,j}\leq m$, we put the star in the box $(i+1,j)$ and the entry $t_{i+1,j}$ in the box $(i,j)$. The other entries are not modified.
\subitem c. if the two boxes $(i+1,j)$ and $(i,j+1)$ do not exist ($(i,j)$ is an inner corner for $T$), we suppress the box $(i,j)$ and the star.
\end{itemize}

At the end of the super jeu de taquin, we get a new skew tableau $T_c=sjdt_c(T)$, the number of boxes in $T_c$ is (number of boxes in $T$) -1. By construction, $T_c$ is a semi standard skew tableau.

\begin{exple}

Let us consider $\mathfrak{sl}(2,3)$ and the following skew semistandard tableaux, with $c=(2,1)$:
$$\aligned
T&=\begin{tabular}{cc|c|c|c}
\cline{3-4}
\hskip 3cm&~~&$1$&$2$&\hskip 3cm\\
\cline{3-4}
&$\star$&$2$&\multicolumn{1}{c}{~~}&\multicolumn{1}{c}{~~~~~~$_2$}\\
\hline
&\multicolumn{1}{|c|}{$3$}&\multicolumn{1}{|c|}{$4$}\\
\cline{2-3}
&\multicolumn{1}{|c|}{$3$}&\multicolumn{1}{|c|}{$5$}\\
\cline{2-3}
&\multicolumn{1}{|c|}{$4$}&\multicolumn{1}{c}{~~}\\
\cline{2-2}
\end{tabular}~~\longrightarrow~~\begin{tabular}{cc|c|c|c}
\cline{3-4}
\hskip 3cm&~~&$1$&$2$&\hskip 3cm\\
\cline{2-4}
&\multicolumn{1}{|c|}{$2$}&\multicolumn{1}{|c|}{$\star$}&\multicolumn{1}{c}{~~}&\multicolumn{1}{c}{~~~~~~$_2$}\\
\hline
&\multicolumn{1}{|c|}{$3$}&\multicolumn{1}{|c|}{$4$}\\
\cline{2-3}
&\multicolumn{1}{|c|}{$3$}&\multicolumn{1}{|c|}{$5$}\\
\cline{2-3}
&\multicolumn{1}{|c|}{$4$}&\multicolumn{1}{c}{~~}\\
\cline{2-2}
\end{tabular}~~\longrightarrow~~\begin{tabular}{cc|c|c|c}
\cline{3-4}
\hskip 3cm&~~&$1$&$2$&\hskip 3cm\\
\cline{2-4}
&\multicolumn{1}{|c|}{$2$}&\multicolumn{1}{|c|}{$4$}&\multicolumn{1}{c}{~~}&\multicolumn{1}{c}{~~~~~~$_2$}\\
\hline
&\multicolumn{1}{|c|}{$3$}&\multicolumn{1}{|c|}{$\star$}\\
\cline{2-3}
&\multicolumn{1}{|c|}{$3$}&\multicolumn{1}{|c|}{$5$}\\
\cline{2-3}
&\multicolumn{1}{|c|}{$4$}&\multicolumn{1}{c}{~~}\\
\cline{2-2}
\end{tabular}\\
&\longrightarrow~~\begin{tabular}{cc|c|c|c}
\cline{3-4}
\hskip 3cm&~~&$1$&$2$&\hskip 3cm\\
\cline{2-4}
&\multicolumn{1}{|c|}{$2$}&\multicolumn{1}{|c|}{$4$}&\multicolumn{1}{c}{~~}&\multicolumn{1}{c}{~~~~~~$_2$}\\
\hline
&\multicolumn{1}{|c|}{$3$}&\multicolumn{1}{|c|}{$5$}\\
\cline{2-3}
&\multicolumn{1}{|c|}{$3$}&\multicolumn{1}{|c|}{$\star$}\\
\cline{2-3}
&\multicolumn{1}{|c|}{$4$}&\multicolumn{1}{c}{~~}\\
\cline{2-2}
\end{tabular}~~\longrightarrow~~\begin{tabular}{cc|c|c|c}
\cline{3-4}
\hskip 3cm&~~&$1$&$2$&\hskip 3cm\\
\cline{2-4}
&\multicolumn{1}{|c|}{$2$}&\multicolumn{1}{|c|}{$4$}&\multicolumn{1}{c}{~~}&\multicolumn{1}{c}{~~~~~~$_2$}\\
\hline
&\multicolumn{1}{|c|}{$3$}&\multicolumn{1}{|c|}{$5$}\\
\cline{2-3}
&\multicolumn{1}{|c|}{$3$}&\multicolumn{1}{c}{~~}\\
\cline{2-2}
&\multicolumn{1}{|c|}{$4$}&\multicolumn{1}{c}{~~}\\
\cline{2-2}
\end{tabular}~~=~~T_c.
\endaligned
$$

Similarly,
$$\aligned
T&=\begin{tabular}{cc|c|c|c}
\cline{3-4}
\hskip 3cm&~~&$1$&$2$&\hskip 3cm\\
\cline{3-4}
&$\star$&$4$&\multicolumn{1}{c}{~~}&\multicolumn{1}{c}{~~~~~~$_2$}\\
\hline
&\multicolumn{1}{|c|}{$3$}&\multicolumn{1}{|c|}{$5$}\\
\cline{2-3}
&\multicolumn{1}{|c|}{$4$}&\multicolumn{1}{|c|}{$5$}\\
\cline{2-3}
&\multicolumn{1}{|c|}{$5$}&\multicolumn{1}{c}{~~}\\
\cline{2-2}
\end{tabular}~~\longrightarrow~~\begin{tabular}{cc|c|c|c}
\cline{3-4}
\hskip 3cm&~~&$1$&$2$&\hskip 3cm\\
\cline{2-4}
&\multicolumn{1}{|c|}{$3$}&\multicolumn{1}{|c|}{$4$}&\multicolumn{1}{c}{~~}&\multicolumn{1}{c}{~~~~~~$_2$}\\
\hline
&\multicolumn{1}{|c|}{$\star$}&\multicolumn{1}{|c|}{$5$}\\
\cline{2-3}
&\multicolumn{1}{|c|}{$4$}&\multicolumn{1}{|c|}{$5$}\\
\cline{2-3}
&\multicolumn{1}{|c|}{$5$}&\multicolumn{1}{c}{~~}\\
\cline{2-2}
\end{tabular}~~\longrightarrow~~\begin{tabular}{cc|c|c|c}
\cline{3-4}
\hskip 3cm&~~&$1$&$2$&\hskip 3cm\\
\cline{2-4}
&\multicolumn{1}{|c|}{$3$}&\multicolumn{1}{|c|}{$4$}&\multicolumn{1}{c}{~~}&\multicolumn{1}{c}{~~~~~~$_2$}\\
\hline
&\multicolumn{1}{|c|}{$4$}&\multicolumn{1}{|c|}{$5$}\\
\cline{2-3}
&\multicolumn{1}{|c|}{$\star$}&\multicolumn{1}{|c|}{$5$}\\
\cline{2-3}
&\multicolumn{1}{|c|}{$5$}&\multicolumn{1}{c}{~~}\\
\cline{2-2}
\end{tabular}\\
&\longrightarrow~~\begin{tabular}{cc|c|c|c}
\cline{3-4}
\hskip 3cm&~~&$1$&$2$&\hskip 3cm\\
\cline{2-4}
&\multicolumn{1}{|c|}{$3$}&\multicolumn{1}{|c|}{$4$}&\multicolumn{1}{c}{~~}&\multicolumn{1}{c}{~~~~~~$_2$}\\
\hline
&\multicolumn{1}{|c|}{$4$}&\multicolumn{1}{|c|}{$5$}\\
\cline{2-3}
&\multicolumn{1}{|c|}{$5$}&\multicolumn{1}{|c|}{$\star$}\\
\cline{2-3}
&\multicolumn{1}{|c|}{$5$}&\multicolumn{1}{c}{~~}\\
\cline{2-2}
\end{tabular}~~\longrightarrow~~\begin{tabular}{cc|c|c|c}
\cline{3-4}
\hskip 3cm&~~&$1$&$2$&\hskip 3cm\\
\cline{2-4}
&\multicolumn{1}{|c|}{$3$}&\multicolumn{1}{|c|}{$4$}&\multicolumn{1}{c}{~~}&\multicolumn{1}{c}{~~~~~~$_2$}\\
\hline
&\multicolumn{1}{|c|}{$4$}&\multicolumn{1}{|c|}{$5$}\\
\cline{2-3}
&\multicolumn{1}{|c|}{$5$}&\multicolumn{1}{c}{~~}\\
\cline{2-2}
&\multicolumn{1}{|c|}{$5$}&\multicolumn{1}{c}{~~}\\
\cline{2-2}
\end{tabular}~~=~~T_c.
\endaligned
$$
\end{exple}

It is possible to describe the extraction procedure by using the super jeu de taquin.\\

\begin{prop}

\

Suppose $(S^+,S^-)$ is a trivial extractable pair, with shape $\mu=(b,b')$ in the semistandard tableau $T=T^+\uplus T^-$, with shape $\lambda=(a,a')$. Consider the semistandard skew tableau $R=T\setminus S$, and apply the super jeu de taquin to $T\setminus S$, starting from the greatest outer corner $c$ in $Out(R)$. We get the skew tableau $R_c$ with shape $(a^c,(a^c)')-(b^c,(b^c)')$.
\begin{itemize}

\item If $c=(i,j)$ and $i>m$, then the star is going each step to the bottom. Especially $a^c=a$, $b^c=b$, and $(a^c)'_k=a'_k$, $(b^c)'_k=b'_k$, except for $k=j$:
$$
(a^c)'_j=a'_j-1,~~(b^c)'_j=b'_j-1,
$$ 
and for $k=j-1$, if $j>1$:
$$
(a^c)'_{j-1}=a'_{j-1}+1,~~(b^c)'_{j-1}=b'_{j-1}+1.
$$

\item if $c=(i,j)$ and $i\leq m$, then the star is going each step to the right. Especially $(a^c)'=a'$, $(b^c)'=b'$, and $a^c_k=a_k$, $b^c_k=b_k$, except for $k=i$:
$$
a^c_i=a_i-1,~~b^c_i=b_i-1,
$$ 
and for $k=i-1$, if $i>1$:
$$
a^c_{i-1}=a_{i-1}+1,~~b^c_{i-1}=b_{i-1}+1.
$$
\end{itemize}
\end{prop}

\begin{proof}

Suppose first the box $c$ is below the line $m$. That means $c=(B'_j,j)$ (as above, $B'_j$ is the height of the tableau $S^-$, if $S=S^+\uplus S^-$) and $B'_{j+1}<B'_j$, suppose moreover that during $k$ moves, the star is moving down, it is therefore in the box $(B'_j+k,j)$. If this box is not an inner corner, the box $(B'_j+k+1,j)$ exists and $B'_{j+1}+k+1\leq B'_j+k+1$, since $S$ is extractable, we have: 
$$
t_{B'_j+k+1,j}<t_{B'_{j+1}+k+1,j+1}\leq t_{B'_j+k+1,j+1},
$$
if the box $(B'_j+k+1,j+1)$ exists. Thus the next move of the star is to the bottom.

Looking at the shape of $R_c$, this proves the first assertion in the proposition. Especially, if $\lambda^c=(a^c,(a^c)')$, $\mu^c=(b^c,(b^c)')$, then $\lambda^c-\mu^c=\lambda-\mu$ is a shape.

The case $c$ above the line $m$ is completely similar.\\
\end{proof}

This proposition means it is possible to write $R_c=T_c\setminus(S_c^+,S_c^-)$ with $r^c_{k,\ell}=t_{k,\ell}$ except, with our notations:
$$\aligned
&\text{if }i>m,&~~&r^c_{k,j}=t_{k+1,j}&~~&(i\leq k<A'_j+m),\\
&\text{if }i\leq m,&~~&r^c_{i,k}=t_{i,k+1}&~~&(j\leq k<A_i),
\endaligned
$$
and $(S_c^+,S_c^-)$ is a trivial pair, with shape $(b^c,(b^c)')$. Remark that this pair is extractable. Indeed the tableau $U$ in Definition \ref{pairextractable} is the same for $(T,(S^+,S^-))$ and $(T_c,(S_c^+,S_c^-))$.

Moreover suppose $Out(R)=\{c_1<\dots<c_p<c'_1<\dots<c'_q\}$, then the greatest element $d$ in $Out(R_c)$ is the following:
\begin{itemize}
\item If $c=c'_q=(i'_q,j'_q)$, then:
$$\aligned
&d=(i'_q,j'_q-1)&~~&\text{if }j'_q>1,\\
&d=(i'_q-1,1)&~~&\text{if }j'_q=1~~\text{ and }~~\sup\{i'_{q-1},m\}<i'_q-1,\\
&d=c'_{q-1}=(i'_{q-1},j'_{q-1})&~~&\text{if }j'_q=1~~\text{ and }~~m<i'_{q-1}=i'_q-1,\\
&d=c_p=(m,j_p)&~~&\text{if }j'_q=1~~\text{ and }~~m=i'_q-1.
\endaligned
$$
\item If $c=c_p=(i_p,j_p)\neq(1,1)$, then:
$$\aligned
&d=(i_p-1,j_p)&~~&\text{if }i_p>1,\\
&d=(1,j_p-1)&~~&\text{if }i_p=1~~\text{ and }~~j_{p-1}<j_p-1,\\
&d=c_{p-1}=(i_{p-1},j_{p-1})&~~&\text{if }i_p=1~~\text{ and }~~j_{p-1}=j_p-1.
\endaligned
$$
\item If $c=c_1=(1,1)$, then $Out(R_c)$ is empty.
\end{itemize}

Now it is possible to repeat the above procedure for $R_c$, if $d$ exists. We get a sequence of successive greatest outer corners denoted $(d_1,\dots,d_k)$ and a sequence of skew tableaux, coming from the super jeu de taquin: $R_0=R$, $R_1=sjdt_{d_1}(R)$,\dots, $R_k=sjdt_{d_k}(R_{k-1})$. If $(S^+,S^-)$ is the maximal trivial extractable pair in $T$, let us put $R_k=maxjdt(R)$.

\begin{prop}

\

$R_k$ is a semistandard tableau. If $(S^+,S^-)$ is the maximal trivial extractable pair in $T$, $R_k=push(T)$, or:
$$
push=maxjdt.
$$
\end{prop}

\begin{proof}

By definition, $Out(R_k)$ is empty, this means the skew tableau $R_k$ is a tableau. We saw it is semistandard.

Let $r^k_{i,j}$ be an entry in $R_k$ such that $i>m$. Remark that all the boxes $(m+1,j),\dots,(m+B'_j,j)$ are in $\{d_1,\dots,d_k\}$, therefore:
$$
r^k_{i,j}=t_{i+B'_j,j}=u_{i,j}.
$$
Similarly, if $i\leq m$,
$$
r^k_{i,j}=t_{i,j+B_i}=u_{i,j}.
$$
Therefore $push=maxjdt$.\\
\end{proof}



\end{document}